\documentclass[11pt,a4paper]{amsart}
\usepackage{amsmath}
\usepackage{mathrsfs}
\usepackage{amssymb}
\usepackage{amsfonts}
\usepackage{amscd}
\usepackage{amsthm}
\usepackage{xypic}
\usepackage{enumerate}
\usepackage{fancyhdr}
\usepackage{color}

\input xy
\xyoption{all}

\newcommand{\A}{\mathbb{A}}

\newcommand{\C}{\mathbb{C}}

\renewcommand{\P}{\mathbb{P}}
\newcommand{\Q}{\mathbb{Q}}
\newcommand{\R}{\mathbb{R}}
\newcommand{\Z}{\mathbb{Z}}

\newcommand{\mmD}{\mathcal{D}}

\newcommand{\mmP}{\mathcal{P}}

\newcommand{\mmZ}{\mathcal{Z}}

\newcommand{\mcH}{\mathcal{H}}

\DeclareMathOperator*{\cl}{cl}

\DeclareMathOperator*{\Div}{div}

\DeclareMathOperator*{\im}{im}

\DeclareMathOperator*{\Spec}{Spec}

\newcommand{\ALT}{\operatorname{Alt}}

\newcommand{\RES}{\operatorname{Res}}

\newcommand{\EE}{{}^{\prime }E}

\newcommand{\IId}{\operatorname{Id}}
\newcommand{\DDV}{\operatorname{div}}

\newcommand{\vist}{\begin{flushright}
$\square$
\end{flushright}}

\numberwithin{equation}{section}
\newtheorem{thm}[equation]{Theorem}
\newtheorem{prop}[equation]{Proposition}
\newtheorem{cor}[equation]{Corollary}
\newtheorem{lemma}[equation]{Lemma}
\newtheorem{conj}[equation]{Conjecture}

\theoremstyle{definition}
\newtheorem{df}[equation]{Definition}
\newtheorem{obs}[equation]{Remark}

\newenvironment{enumerate*}[1][{}]{\begin{itemize}\renewcommand{\labelitemi}{#1}}{\end{itemize}}

\title{On Goncharov's regulator and higher arithmetic Chow groups}
\author{J. I. Burgos Gil}
\address{Instituto de ciencias matemáticas CSIC-UAM-UC3M-UCM, Spain}
\email{jiburgosgil@gmail.com}

\author{E. Feliu}
\address{Facultat de Matem\`atiques, Universitat de Barcelona, Spain}
\email{efeliu@ub.edu}

\author{Y. Takeda}
\address{Faculty of Mathematics, Kyushu University, Japan}
\email{yutakeda@math.kyushu-u.ac.jp}

\date{\today}
\thanks{This work was partially supported by the project MTM2006-14234-C02-01
and by the JSPS Grant-in-Aid for Scientific Research (No.17340009)}

\newif\ifprivate
\privatetrue

\begin{document}

\begin{abstract}
In this paper we show that the regulator defined by Goncharov in
\cite{Goncharov} from higher algebraic Chow groups to
Deligne-Beilinson cohomology agrees with Beilinson's regulator. We
give a direct comparison of Goncharov's regulator to the construction
given by Burgos and Feliu in \cite{BFChow}. As a consequence, we show
that the higher arithmetic Chow groups defined by Goncharov agree, for
all projective arithmetic varieties over an arithmetic field, with the
ones defined by Burgos and Feliu.

\end{abstract}

\maketitle

\section*{Introduction}
Let $X$ be an arithmetic variety over a field, i.e. a regular scheme
which is flat and quasi-projective over an arithmetic field. Assuming
that $X$ is projective, Goncharov introduced in \cite{Goncharov} the
higher arithmetic Chow groups of $X$, $\widehat{CH}^p(X,n)$, which fit
in a long exact sequence of the form
\begin{align*}
  \cdots \rightarrow \widehat{CH}^{p}(X,n) \xrightarrow{\zeta}
  CH^{p}(X,n) \xrightarrow{\rho} H_{\mmD}^{2p-n}(X,\R(p))
  \xrightarrow{a} \widehat{CH}^{p}(X,n-1) \rightarrow \cdots \\ \cdots
  \rightarrow CH^p(X,1) \xrightarrow{\rho} \mmD^{2p-1}(X,p) /
  \im d_{\mmD} \xrightarrow{a} \widehat{CH}^p(X) \xrightarrow{\zeta}
  CH^p(X) \rightarrow 0.
\end{align*}
Here, $\widehat{CH}^p(X)$ denote
the arithmetic Chow groups defined by Gillet and Soul\'e in
\cite{GilletSouleIHES}, $H_{\mmD}^{2p-*}(X,\R(p))$ are the real
Deligne-Beilinson cohomology groups and $(\mmD^{\ast}(X,\ast),
  d_{\mmD})$ is the Deligne complex of differential forms.

The groups $\widehat{CH}^p(X,n)$ are obtained as the homology groups
of
the simple complex associated to a regulator morphism
$$ Z^p(X,*) \xrightarrow{\mmP}\  \tau \mathcal {D}_{D}^{2p-*}(X,p),$$
from the chain complex $Z^p(X,*)$ computing the higher algebraic Chow
groups of Bloch, to the Deligne complex of currents $ \tau \mathcal
{D}_{D}^{2p-*}(X,p)$.

Goncharov's definition left open the question whether the composition
of the isomorphism $K_n(X)_{\Q}\cong \bigoplus_{p\geq 0}
CH^p(X,n)_{\Q}$ given in \cite{Bloch1} with the morphism induced by
$\mmP$ agrees with
Beilinson's regulator. In addition, the possibility of defining pull-back
morphisms and a product structure on $\bigoplus_{p,n}
\widehat{CH}^{p}(X,n)$ was also left open.

Later, in \cite{BFChow}, Burgos and Feliu introduced a new definition
of the higher arithmetic Chow groups, suitable for quasi-projective
arithmetic varieties over a field. The main difference with
Goncharov's construction was the use of the Deligne complex of
differential forms with logarithmic singularities instead of the
Deligne complex of currents. This enabled one to have well-defined
pull-backs and a product structure on $\bigoplus_{p,n}
\widehat{CH}^{p}(X,n)$. The new definition was as well based on the
definition of a regulator, which was shown to induce Beilinson's
regulator in $K$-theory.

In loc. cit., it was left a comparison of the two definitions of
higher arithmetic Chow groups. In this paper, we show that both
definitions agree in the
case of proper arithmetic varieties over an arithmetic field. This is
shown by a direct comparison of  Goncharov's regulator and the one
introduced by Burgos and Feliu in \cite{BFChow}.

\medskip

The paper is organized as follows. The first four sections contain the
required preliminaries. The first section contains the notation
that will be used in the sequel. Section 2 and 3 cover the necessary
background on Deligne-Beilinson cohomology and Bloch higher algebraic
Chow groups respectively.  In section 4 we review the construction of the
regulator given by Burgos and Feliu.  From section 5 to section 7, we find the
core of this work. In section 5 we introduce the differential forms $T_m$ and a few properties are shown.
A comparison of Wang's forms and the differential forms given by Goncharov in his construction of  the regulator is performed.
In section 6, the comparison of regulators is done and finally, in section 7 we
prove that the higher arithmetic Chow groups given by Goncharov and the ones given by Burgos and Feliu agree.

\section{Notation}

\subsection{Notation on (co)chain complexes}
We use the standard conventions on (co)chain complexes. By a (co)chain
complex we mean a (co)chain complex over the category of abelian
groups.  The \emph{cochain complex associated to a chain complex}
$A_*$ is simply denoted by $A^*$ and the \emph{chain complex
  associated to a cochain complex} $A^*$ is denoted by $A_*$.

The \emph{translation of a cochain complex} $(A^*,d_A)$ by an integer $m$ is denoted by $A[m]^*$.
Note that $A[m]^{n}=A^{m+n}$ and the differential of $A[m]^*$ is
$(-1)^m d_A$. If $(A_*,d_A)$ is a \emph{chain complex}, then the
translation of $A_*$ by an integer $m$ is denoted by $A[m]_*$. In
this case the differential is also $(-1)^m d_A$ but
$A[m]_{n}=A_{n-m}$.

The \emph{simple complex} associated to an iterated chain complex
$A_*$ is denoted by $s(A)_*$ and the analogous notation is used for
the simple complex associated to an iterated cochain complex (see
\cite{BurgosKuhnKramer}, $\S$2 for definitions).  The \emph{simple of a
  cochain map} $A^*\xrightarrow{f} B^*$ is the cochain complex
$(s(f)^*,d_s)$ with $s(f)^n=A^n\oplus B^{n-1}$, and differential
$d_s(a,b)=(d_A a,f(a)-d_Bb)$.  Note that this complex is the cone of
$-f$ shifted by 1. There is an associated long exact sequence
\begin{equation}\label{longsimple}
\dots \rightarrow H^n(s(f)^*) \rightarrow H^n(A^*) \xrightarrow{f}
H^n(B^*) \rightarrow H^{n+1}(s(f)^*) \rightarrow \cdots
\end{equation}
%If $f$ is surjective, there is a quasi-isomorphism
%\begin{equation}\label{simplekernel}
%  \ker f  \xrightarrow{i}  s(-f)^*   \qquad  x \mapsto  (x,0),
%\end{equation}
%and if $f$ is injective, there is a quasi-isomorphism
%\begin{equation}\label{simplequotient}
%  s(f)[1]^* \xrightarrow{\pi} B^*/A^*   \qquad  (a,b) \mapsto  b.
%\end{equation}
Equivalent results can be stated for chain complexes.

Following Deligne \cite{DeligneHodgeII}, given a cochain complex $A^*$
and an integer $n$, we denote by $\tau_{\leq n}A^*$, $\tau_{\geq
  n}A^*$ the \emph{canonical truncations of $A^*$ at degree $n$}.

\subsection{Cubical abelian groups and chain complexes}\label{cubical}
Let $C_{\cdot}=\{C_n\}_{n\geq 0}$ be a cubical abelian group. We
denote the face maps by $\delta_i^j: C_n\rightarrow C_{n-1}$, for
$i=1,\dots,n$ and $j=0,1$, and the degeneracy maps by $\sigma_i:
C_n\rightarrow C_{n+1}$, for $i=1,\dots,n+1$. Let $D_n\subset C_n$ be
the subgroup of \emph{degenerate elements} of $C_n$,

By $C_*$ we denote the \emph{associated chain complex}, that is, the
chain complex whose $n$-th graded piece is $C_n$ and whose
differential is given by $\delta= \sum_{i=1}^n
\sum_{j=0,1}(-1)^{i+j}\delta_i^j.$ We fix the \emph{normalized chain
  complex} associated to $C_{\cdot}$, $NC_*$, to be the chain complex
whose $n$-th graded group is $NC_n := \bigcap_{i=1}^{n} \ker
\delta_i^1,$ and whose differential is $ \delta=\sum_{i=1}^n
(-1)^{i}\delta_i^{0}.$ It is well-known that there is a decomposition
of chain complexes $C_* \cong NC_* \oplus D_* $ giving an isomorphism
of chain complexes $NC_* \cong C_*/D_*.$

\section{Deligne-Beilinson cohomology}
\label{dbcohomology}\label{deligne}

%\subsection{Deligne-Beilinson cohomology }
In this paper we use the definitions and conventions on
Deligne-Beilinson cohomology given in \cite{Burgos2} and
\cite{BurgosKuhnKramer}, $\S$5.

As defined in \cite{BurgosKuhnKramer}, $\S$5.2, for any Dolbeault
complex $A = (A_{\R}, d_A)$ there is an associated cochain complex
called the \emph{Deligne complex} and denoted by
$(\mmD^*(A,*),d_{\mmD})$.

Let $X$ be a complex algebraic manifold.  Let
$E_{\mathbb R}^n(X)$ (resp. $E_{\mathbb R,c}^n(X)$) be the space of
real smooth differential forms (resp. differential forms with compact
support) of degree $n$ on $X$, and let $\EE_{\mathbb R}^n(X)$ be the
space of real currents on $X$ of degree $n$, that is, the topological
dual of $E_{\mathbb R,c}^{-n}(X)$. One denotes $\R(p)=(2\pi i)^p
\cdot \R \subset \C$. Accordingly, we write $E_{\mathbb R}^n(X,p)=(2\pi
i)^{p}\cdot E_{\mathbb R}^n(X)$ and $\EE_{\mathbb R}^n(X,p)=(2\pi
i)^{p}\cdot \EE_{\mathbb R}^n(X)$.

When $X$ is equidimensional of dimension $d$, we write
\begin{equation}\label{eq:9}
  D_{\mathbb R}^{\ast}(X)=\EE_{\mathbb{R}}^{\ast}(X)[-2d](-d).
\end{equation}
Hence $D_{\mathbb{R}}^{n}(X,p)$ is the topological dual of $E_{\mathbb
  R}^{2d-n}(X,d-p)$.

\subsection{Deligne complex of differential forms}
Let $(E_{\log,\R}^*(X),d)$ be the complex of \emph{real differential
  forms with logarithmic singularities along infinity} \cite{Burgos3} and let
$E_{\log,\R}^*(X)(p)$ denote the vector space of real differential
forms with logarithmic singularities along infinity, twisted by $p$.
Since the complex $(E_{\log,\R}^*(X),d)$ is a Dolbeault complex (see
\cite{BurgosKuhnKramer}, $\S$5.2) we can consider the Deligne complex
of differential forms with logarithmic singularities
$$(\mathcal{D}_{\log}^*(X,p),d_{\mmD}):=(\mathcal{D}^*(E_{\log,\R}^*(X),p),d_{\mmD}).$$
This complex is functorial on $X$.  It computes the real
Deligne-Beilinson cohomology of $X$, that is, there is an isomorphism
$$H^n(\mmD^*_{\log}(X,p)) \cong H_{\mmD}^n(X,\R(p)).$$
Morever, the Deligne-Beilinson cohomology product structure can be
described by a cochain morphism on the Deligne complex (see
\cite{Burgos2})
$$ \mmD^n(X,p) \otimes \mmD^m(X,q)  \xrightarrow{\bullet}  \mmD^{n+m}(X,p+q). $$
This product is graded commutative and satisfies the Leibniz rule, but
it is only associative up to homotopy.

If $X$ is compact, then we simply denote by $\mmD^*(X,p)$ the Deligne
complex of differential forms on $X$.

\subsection{Currents} \label{sec:currents}
Assume that $X$ is compact and equidimensional of dimension $d$.
The complex $D_{\mathbb{R}}^{\ast}(X)$ has also a structure of Dolbeault
complex, and hence, there is an associated Deligne complex denoted by
$$(\mathcal{D}_{D}^*(X,p),d_{\mmD}):=(\mathcal{D}^*(D_{\R}^*(X),p),d_{\mmD}).
$$

The current associated to every differential form gives a
quasi-isomorphism of Deligne complexes
\begin{equation}
  \label{eq:14}
\mathcal{D}^*(X,p)\xrightarrow{[\cdot]}
\mathcal{D}_{D}^*(X,p),\qquad \alpha \mapsto [\alpha],
\end{equation}
where $[\alpha ]$ is the current
\begin{displaymath}
  [\alpha ](\omega )=\frac{1}{(2\pi i)^{d}} \int_{X}\omega \land \alpha.
\end{displaymath}

Therefore, the cohomology groups of $\mathcal{D}_{D}^*(X,p)$
are isomorphic to the Deligne-Beilinson cohomology groups
of $X$:
$$
H^n(\mathcal{D}_{D}^*(X,p))\cong H_{\mathcal D}^n(X, \mathbb R(p)).
$$

Recall that we are using the conventions of \cite{BurgosKuhnKramer}
concerning the twisting and the real structures. In particular, if $Y$
is a subvariety of codimension $p$ then the current integration along
$Y$, denoted $\delta _{Y}$ is given by
\begin{displaymath}
  \delta _{Y}(\omega )=\frac{1}{(2\pi i)^{d-p}}\int _{\widetilde
    Y}\iota ^{\ast}\omega ,
\end{displaymath}
for $\iota :\widetilde Y\longrightarrow X$ a resolution of
singularities of $Y$. Hence, again using the conventions of
\cite{BurgosKuhnKramer} we have $\delta _{Y}\in
\mathcal{D}_{D}^{2p}(X,p)$ and is a representative of the class
$\cl(Y)\in H^{2p}_{\mathcal{D}}(X,p)$.

\section{Bloch Higher Chow groups}\label{higherchow}
We recall here the definition of \emph{higher algebraic Chow groups}
given by Bloch in \cite{Bloch1}.  Initially, they were defined using
the chain complex associated to a simplicial abelian group. An
alternative definition can be given using a cubical presentation, as
developed by Levine in \cite{Levine1}. Both constructions are
analogous, however, the cubical setting is more suitable to define
products.  These two settings are recalled here.

Let $k$ be a field and $\mathbb P^n$ the projective space of dimension
$n$ over $k$.  Fix $X$ to be an equidimensional quasi-projective
algebraic scheme of dimension $d$ over the field $k$.

\subsection{The simplicial Bloch complex}\label{simplicial}
We recall first the definition of the higher Chow groups using a
simplicial setting as given by Bloch in \cite{Bloch1}.  We fix
homogeneous coordinates $z_0,\ldots,z_n$ of $\mathbb P^n$ and put
$\Delta^n=\mathbb P^n\setminus H_n,$ where $H_n$ is the hyperplane defined by
$z_0+\cdots +z_n=0$.
%$z_0=0$.
The collection $\{\Delta^n\}_{n\geq 0}$ has a cosimplicial scheme
structure, with coface maps denoted by $\partial^i$.
%. given, for $i=0,\dots,n$, by
%\begin{align*}
%\delta^i(z_0,\dots,z_n)=& (z_0,\dots,z_{i-1},0,z_i,\dots,z_{n}), \\
%\sigma^i(z_0,\dots,z_m)=&(z_0,\dots,z_i+z_{i+1},\dots,z_m).
%\end{align*}
%Then $(\Delta^m, \delta^i, \sigma^i)$ forms a cosimplicial scheme.
The faces of $\Delta^n$ are closed subschemes which arise as the image
of compositions of coface maps. A face of $X\times \Delta ^{n}$ is a
closed subscheme of the form $X\times F$, for $F$ a face of $\Delta
^{n}$.

We denote by $Z_s^p(X,n)$ the free abelian group generated
by the codimension $p$ closed irreducible
subvarieties of $X\times \Delta^{n}$, which intersect properly
all the faces of $X\times \Delta^n$.

Then, intersection with the face $\partial^i(X\times \Delta^{n-1})$
gives a map $\partial_i:Z_s^p(X, n)\to Z_s^p(X, n-1)$.
Setting
$$
\partial =\sum_{i=0}^n(-1)^i\partial_i:Z_s^p(X,n)\longrightarrow
Z_s^p(X,n-1),
$$
then $(Z_s^p(X, *), \partial )$ is a chain complex of abelian groups.
The \emph{higher algebraic Chow groups} of $X$ are the homology groups of this
complex:
$$
CH^p(X,n)=H_n(Z_s^p(X, *)).
$$
To emphasize explicitly that this higher algebraic Chow groups are given using the simplicial setting, we will write $$CH_{s}^p(X,n)=H_n(Z_s^p(X, *)).$$

\subsection{The cubical Bloch complex}\label{cubicalp1}
Consider $\P^1$ the projective line over $k$ and let $ \square =
\P^1\setminus \{1\}\cong \A^1.$ The collection
$\{\square^{n}=\square\times\overset{n}{\dots}\times\square\}_{n\ge
  0}$
has a cocubical scheme structure, with coface maps
denoted by $\delta^i_0,\delta_1^i$ and codegeneracy maps denoted by
$\sigma ^{i}$. Specifically, for $i=1,\dots,n$, the coface and codegeneracy maps
are defined as
\begin{eqnarray*}
 \delta_0^i(x_1,\dots,x_{n-1}) &=& (x_1,\dots,x_{i-1},0,x_{i},\dots,x_{n-1}), \\
\delta_1^i(x_1,\dots,x_{n-1}) &=&
(x_1,\dots,x_{i-1},\infty,x_{i},\dots,x_{n-1}), \\
\sigma^i(x_1,\dots,x_n) &=& (x_1,\dots,x_{i-1},x_{i+1},\dots,x_n).
\end{eqnarray*}
Note that the cofaces are closed immersions and the codegeneracies are
flat maps.

An $r$-dimensional \emph{face} of $\square^n$ is any closed subscheme
of the form
$\delta^{i_1}_{j_1}\cdots \delta^{i_r}_{j_r}(\square^{n-r})$. Faces of
$X\times \square^{n}$ are defined accordingly.

Let $Z^{p}_c(X,n)$ be the free abelian group generated by the
codimension $p$ closed irreducible subvarieties of $X\times
\square^{n}$, which intersect properly all the faces of $X\times
\square^n$.  The pull-back by the coface and codegeneracy maps of
$\square^{\cdot}$ endow $Z^{p}_c(X,\cdot)$ with a cubical abelian
group structure. Let $(Z^{p}_c(X,*),\delta)$ be the associated chain
complex and consider the \emph{normalized chain complex} associated to
$Z^p_c(X,*)$, denoted by $Z^p_c(X,*)_0$ with
$$Z^{p}_c(X,n)_{0}:= NZ^p_c(X,n)=\bigcap_{i=1}^{n} \ker \delta^1_i.$$

Denote the homology groups of this complex by
$$CH_{c}^p(X,n) =H_n(Z^{p}_c(X,*)_{0}).$$
The subscript $c$ refers to the use of the cubical setting.
The \emph{higher Chow groups} defined by Bloch are isomorphic to
$CH_{c}^p(X,n)$. That is, there is a natural isomorphism
(\cite{Levine1}, Theorem 4.7):
\begin{displaymath}
  CH^p_c(X,n)\cong CH^p_s(X,n),\qquad \forall n,p\geq 0.
\end{displaymath}

\section{Burgos-Feliu construction of the Beilinson
  regulator}\label{BFregulator}
We review here briefly  the construction of the Beilinson
regulator given by Burgos and Feliu in \cite{BFChow}.

Let $X$ be a complex algebraic variety and let $\square$ be as
in $\S$\ref{cubicalp1}. For simplicity we will assume that $X$ is proper although in
 \cite{BFChow} the construction is  also done for open varieties.
Consider the smooth compactifications of $X\times \square^n$ given by
$X\times (\P^1)^n$. We denote $D^{n}=X\times (\P^1)^n
\setminus X \times \square^n$, which is a normal crossing divisor. Let
$E^*_{X\times (\P^1)^n }(\log D^{n})$ be the complex of
\emph{differential forms with logarithmic singularities along $D$}
\cite{Burgos3}.

In this paper we will denote
$$E_{\log}^*(X\times \square^n) := E^*_{X\times (\P^1)^n}(\log D^{n})$$
and we will call it
the complex of \emph{differential forms on $X\times \square^n$ with
  logarithmic singularities along infinity}.
Then, the Deligne complex associated to $E_{\log}^*(X\times
\square^n)$, denoted $\mmD_{\log}^{\ast}(X\times \square^n,p)$, computes the
Deligne-Beilinson cohomology of $X\times\square ^{n}$, which, by
homotopy invariance, agrees with the Deligne-Beilinson cohomology of
$X$.

\begin{obs}
  The Deligne complex of differential forms on $X\times \square^n$ with
  logarithmic
  singularities along infinity is usually defined by taking the limit
  over all compactifications of $X\times \square^n$. In this work,
  however, since $X$ is proper, we have a natural compactification of
  $X\times \square^n$. The
  two different cochain complexes obtained, using the limit or with a
  fixed compactification,  are quasi-isomorphic.
\end{obs}

The description of the regulator uses some intermediate complexes that
we describe in the following.

\medskip
\textbf{The complex $\mmD_{\A}^{*}(X,p)_0$. } For every $n,p\geq 0$,
let $\tau\mathcal{D}_{\log}^*(X\times \square^n,p)$ be the Deligne
complex of differential forms in $X\times \square^n$, with logarithmic
singularities at infinity, truncated at degree $2p$:
\begin{displaymath}
  \tau\mathcal{D}_{\log}^*(X\times \square^n,p):=
  \tau_{ \le 2p} \mathcal{D}_{\log}^*(X\times \square^n,p).
\end{displaymath}
The structural
maps of the cocubical scheme $\square^{\cdot}$ induce a cubical
structure on $\tau\mathcal{D}_{\log}^r(X\times \square^*,p)$ for every
$r$ and $p$.  Consider the $2$-iterated cochain complex given by
$$\mathcal{D}_{\A}^{r,-n}(X,p)=\tau\mathcal{D}_{\log}^r(X\times \square^n,p)$$
and with differential $(d_{\mmD},\delta=\sum_{i=1}^n
(-1)^i(\delta_i^0-\delta_i^1))$.
Let
$$\mmD_{\A}^*(X,p)=s(\mathcal{D}_{\A}^{*,*}(X,p))$$
be the simple complex associated to the $2$-iterated complex
$\mathcal{D}_{\A}^{*,*}(X,p)$.

For every $r,n$, let $\tau\mmD_{\log}^r(X\times \square^n,p)_0=
N\tau\mmD_{\log}^r(X\times \square^n,p)$ be the normalized complex and let
$$\mmD_{\A}^{r,-n}(X,p)_0= \tau\mmD_{\log}^r(X\times \square^n,p)_0.$$
Denote by $(\mmD^*_{\A}(X,p)_0,d_s)$ the associated simple complex.

\medskip
\textbf{The complex $\mcH^p(X,*)_0$. }
Let $\mathcal{Z}^p_{n,X}$ be the set of all codimension $p$ closed
subvarieties of $X\times \square^n$ intersecting properly the
faces of $X\times \square^n$.
When there is no source of confusion, we
simply write $\mmZ^p_n$ or even $\mmZ^p$.
Write
\begin{displaymath}
  \mmD_{\log}^{*}(X\times \square^n\setminus \mathcal{Z}_n^p,p)=
  \lim_{\substack{\longrightarrow\\Z\in \mathcal{Z}_n^p}} \mmD_{\log}^{*}(X\times
  \square^n\setminus Z,p),
\end{displaymath}
where here logarithmic singularities at infinity are defined by taking the
limit over all possible compactifications.

Let $(\mmD_{\log,\mathcal{Z}_n^p}^{*}(X\times \square^n,p),d_{\mmD})$
be the Deligne complex with supports
\begin{displaymath}
  \mmD_{\log,\mathcal{Z}_n^p}^{*}(X\times \square^n,p)=
  s(\mmD_{\log}^{\ast}(X\times \square^n,p)\rightarrow
  \mmD_{\log}^{*}(X\times \square^n\setminus \mathcal{Z}_n^p,p)).
\end{displaymath}
The cohomology groups of this complex are denoted by
$H^{*}_{\mmD,\mathcal{Z}^p_n}(X\times \square^n,\R(p))$.

Consider the cubical
abelian group
\begin{equation}
\mcH^p(X,\cdot):=H^{2p}_{\mmD,\mathcal{Z}^p_\cdot}(X\times
\square^\cdot,\R(p)),
\end{equation}
with faces and degeneracies induced by those of $\square^{\cdot}$. Let
$\mcH^p(X,*)_0$ be the associated normalized complex.

\medskip
\textbf{The complex $\mmD^{*,*}_{\A,\mathcal{Z}^p}(X,p)_0$. } Let
$\mmD^{*,*}_{\A,\mathcal{Z}^p}(X,p)_0$ be the $2$-iterated cochain
complex, whose component of bidegree $(r,-n)$ is
$$\mmD^{r,-n}_{\A,\mathcal{Z}^p}(X,p)_0:=
\tau_{\leq 2p}\mmD_{\log,\mathcal{Z}_n^p}^{r}(X\times
\square^n,p)_0=N\tau_{\leq 2p}\mmD_{\log,\mathcal{Z}_n^p}^{r}(X\times
\square^n,p), $$ and whose
differentials are
 $(d_{\mmD},\delta)$. As usual,
we denote by $(\mmD^{*}_{\A,\mathcal{Z}^p}(X,p)_0,d_s)$ the associated
simple complex. Let $\mmD^{2p-*}_{\A,\mathcal{Z}^p}(X,p)_0$ be the chain complex
whose $n$-th graded piece is $\mmD^{2p-n}_{\A,\mathcal{Z}^p}(X,p)_0$.

The main properties of the above complexes are summarized in the
following result.

\begin{prop}[\cite{BFChow}]\label{difaffine}
\begin{enumerate}
\item \label{item:8} The natural morphism of complexes
$$\tau\mmD^*_{\log}(X,p)= \mmD_{\A}^{*,0}(X,p)_0 \rightarrow
\mmD_{\A}^{*}(X,p)_0$$
is a quasi-isomorphism.
\item \label{item:9} There is an isomorphism of chain complexes
$$
f_1: Z^p_c(X,*)_{0}\otimes \R  \xrightarrow{\cong}  \mcH^p(X,*)_0,
$$
sending every algebraic cycle $z$ to its class $\cl(z)$.
\item \label{item:10} For every $p\geq 0$, the morphism
\begin{eqnarray*}
\mmD^{2p-n}_{\A,\mathcal{Z}^p}(X,p)_0 & \xrightarrow{g_1}&
\mcH^{p}(X,n)_0 \\
((\omega_n,g_n),\dots,(\omega_0,g_0))  & \mapsto &
[(\omega_n,g_n)]
\end{eqnarray*}
defines a quasi-isomorphism of chain complexes.
\end{enumerate}
\end{prop}
\begin{proof}
  See \cite{BFChow}, Corollary 2.9, Lemma 2.11 and Proposition 2.13.
\end{proof}

\textbf{Definition of the regulator. }
Consider the map of iterated cochain complexes defined by the
projection onto the first factor
\begin{displaymath}
\begin{array}{rccc}
\mmD_{\A,\mathcal{Z}^p}^{r,-n}(X,p)=&\tau_{\le 2p}s^{r}\left(\mmD_{\log}^\ast(X\times
\square^n,p)\rightarrow \mathcal{D}_{\log}^{*}(X\times
\square^n\setminus \mmZ^p,p)\right) & \xrightarrow{\rho} &
\tau\mmD_{\log}^r(X\times \square^n,p)
\\ &(\alpha,g) & \mapsto & \alpha.
\end{array}
\end{displaymath}
It induces  a chain morphism
\begin{equation}\label{regulator4}
 \mmD_{\A,\mathcal{Z}^p}^{2p-*}(X,p)_0 \xrightarrow{\rho}
\mmD_{\A}^{2p-*}(X,p)_0.\end{equation} The morphism induced by
$\rho$ in homology, together with the isomorphisms of Proposition
\ref{difaffine}
 induce a morphism
\begin{equation}\label{regulator}\rho: CH^p(X,n) \rightarrow
  CH^p(X,n)_{\R} \rightarrow
H_{\mmD}^{2p-n}(X,\R(p)).\end{equation} By abuse of notation, all
these morphisms are denoted by $\rho$.

Observe that in the derived category of chain complexes, the morphism
$\rho$ is given by
the composition
$$
\xymatrix{
Z^{p}_c(X,*)_{0} \ar[r]^{f_1} & \mcH^p(X,*)_0 &  \ar[l]_{g_1}^{\sim}
\mmD^{2p-*}_{\A,\mathcal{Z}^p}(X,p)_0 \ar[r]^{\rho} & \mmD_{\A}^{2p-*}(X,p)_0. }
$$

The following result follows directly from the definitions.
\begin{lemma}\label{lemm:1}
Let $z\in
CH^p(X,n)$, then
$$\rho(z) =  [(\omega _n,\dots,\omega _0)],$$
for any cycle $((\omega _n,g_n),\dots,(\omega _0,g_0))\in
\mmD_{\A,\mathcal{Z}^p}^{2p-n}(X,p)_0$  such that
$[(\omega _n,g_n)]=\cl(z)$.
\ \hfill $\square$
\end{lemma}

\begin{thm}[\cite{BFChow}, Thm. 3.5]\label{beichow} Let $X$ be an
  equidimensional complex algebraic manifold.
Let $\rho'$ be the composition of $\rho$ with the isomorphism
given by the Chern character of \cite{Bloch1}
$$\rho': K_n(X)_{\Q} \xrightarrow{\cong} \bigoplus_{p\geq 0}
CH^{p}(X,n)_{\Q} \xrightarrow{\rho}
 \bigoplus_{p\geq 0}H^{2p-n}_{\mmD}(X,\R(p)).$$
Then, the morphism $\rho'$ agrees with the Beilinson regulator.
\hfill $\square$
\end{thm}

We consider now the diagram of chain complexes
 {\small $$
C_{\ast}=\left(\begin{array}{c}\xymatrix@C=0.5pt{
 \mcH^p(X,*)_0 & & \mmD_{\A}^{2p-*}(X,p)_0 \\
 & \ar[ul]_{g_1}^{\sim}
\mmD^{2p-*}_{\A,\mathcal{Z}^p}(X,p)_0 \ar[ur]^{\rho} }
\end{array}\right). $$ }
We denote
\begin{equation}
  \label{eq:18}
  \mathcal{D}_{\A,\mathcal{H}}^{2p-\ast}(X,p)_{0}=s(C)[-1]_*
\end{equation}
the simple complex associated to the above diagram as in \cite{BFChow} $\S$1.2, shifted by
 minus one. That is, an
element of $\mathcal{D}_{\A,\mathcal{H}}^{2p-n}(X,p)_{0}$ is a triple
$(\alpha _{1},\alpha _{2},\alpha _{3})$ with $\alpha _{1}\in \mmD^{2p-n-1}_{\A,\mathcal{Z}^p}(X,p)_0$, $\alpha _{2}\in \mcH^p(X,n)_0$ and $\alpha _{3}\in
\mmD_{\A}^{2p-n}(X,p)_0$, and the differential is given by
\begin{displaymath}
  d(\alpha _{1},\alpha _{2},\alpha _{3})=(-d\alpha _{1},d\alpha _{2}+g_{1}(\alpha
  _{1}),d\alpha _{3}-\rho (\alpha _{1})).
\end{displaymath}
There is a quasi-isomorphism
\begin{displaymath}
  \beta : \mmD_{\A}^{2p-*}(X,p)_0 \xrightarrow{\sim}
  \mathcal{D}_{\A,\mathcal{H}}^{2p-\ast}(X,p)_{0}
\end{displaymath}
given by $\beta (\alpha )=(0,0,\alpha )$. We identify
\begin{displaymath}
  H_{\ast}(\mathcal{D}_{\A,\mathcal{H}}^{2p-\ast}(X,p)_{0}) =
  H_{\mathcal{D}}^{2p-\ast}(X,\R(p))
\end{displaymath}
by means of this quasi-isomorphism. Then Beilinson's regulator for
higher Chow groups is given by the morphism, also denoted $\rho $,
\begin{displaymath}
  \rho :Z^{p}_c(X,*)_{0}\longrightarrow
  \mathcal{D}_{\A,\mathcal{H}}^{2p-\ast}(X,p)_{0}
\end{displaymath}
given by $\rho (z)=(f_{1}(z),0,0)$.

\section{Goncharov's forms and Wang's forms}

\subsection{The differential form $\mathbf{T_m}$}
Let $A^{\ast}$ be a Dolbeault algebra and let $\mathcal{D}^{\ast}(A,\ast)$ be the
associated Deligne algebra as given in \cite{BurgosKuhnKramer}.
Let $u_{1},\dots,u_{m}\in
\mathcal{D}^{1}(A,1)$. Following \cite{Wang}, for $i=1,\dots,m$, we
write
{\small \begin{equation} \label{eq:12}
  S_{m}^{i}(u_{1},\dots,u_{m})=(-2)^{m}\sum_{\sigma \in \mathfrak S_m}
(-1)^{|\sigma|}u_{\sigma (1)}\partial u_{\sigma (2)}\land\dots\land
\partial u_{\sigma (i)}\land \bar \partial u_{\sigma
  (i+1)}\land\dots\land \bar \partial u_{\sigma({m})},
\end{equation}}
and
\begin{equation} \label{eq:13}
  T_{m}(u_{1},\dots,u_{m})=
  \frac{1}{2m!}\sum_{i=1}^{m}(-1)^{i}S_{m}^{i}(u_{1},\dots u_{m}).
\end{equation}
We will also write $T_{0}=1$.
The forms  $T_{m}(u_{1},\dots,u_{m})$ will be called Wang's forms.
\begin{prop}\label{prop:pd}
  \begin{enumerate}
  \item \label{item:1}The form $ T_{m}(u_{1},\dots,u_{m})$ belongs to
    $\mathcal{D}^{m}(X,m)$.
  \item \label{item:2}It holds
    \begin{displaymath}
      T_{m}(u_{1},\dots,u_{m})=\frac{1}{m!}\sum_{\sigma \in
        \mathfrak S_m} (-1)^{|\sigma|}u_{\sigma (1)}\bullet (
      \dots \bullet (u_{\sigma (m-1)}\bullet u_{\sigma (m)})).
    \end{displaymath}
  \item \label{item:3}There is a recursive formula
    \begin{equation}\label{eq:3}
      d_{\mathcal{D}}T_{m}(u_{1},\dots,u_{m})=
      \sum_{i=1}^{m}(-1)^{i-1}d_{\mathcal{D}}u_{i}\bullet T_{m-1}(u_{1},\dots,
      \widehat {u_{i}},\dots ,u_{m}).
    \end{equation}
  \end{enumerate}
\end{prop}
\begin{proof}
  It is clear that \ref{item:2} implies \ref{item:1} and, by the
  Leibniz rule, also \ref{item:3}. Nevertheless we will prove first
  \ref{item:3} and use it to prove \ref{item:2}. For this we will
  follow \cite{Takeda} $\S$5.2, but note that our $S_{m}^{i}$ is
  $(-2)^{m}$ times the form denoted $S_{m}^{i}$ there. Given elements
  $u_{1},\dots,u_{n}\in \mathcal{D}^{1}(A,1)(=A^{0}_{\R}(0))$, we
  will denote by $(u_{1},\dots,u_n)^{(i)}$ the piece of bidegree
  $(i,n-i)$ of $du_{1}\land\dots \land du_{n}\in A^{n}$. Then
  \begin{displaymath}
    S^{i}_{m}(u_{1},\dots,u_m)=(-2)^{m}(i-1)!(m-i)!
    \sum_{j=1}^{m}(-1)^{j+1}u_{j}(u_{1},\dots,\widehat{u_{j}},
    \dots,u_{m})^{(i-1)}.
  \end{displaymath}
  It is proved in \cite{Takeda}, Lemma 5.3  that
  \begin{multline}
    \label{eq:2}
    \partial
    S^{i}_{m}(u_{1},\dots,u_m)=(-2)^{m}i!(m-i)!(u_{1},\dots,u_{m})^{(i)}\\
    +(m-i)\sum_{j=1}^{m}
    (-1)^{j}(-2\partial\bar \partial)u_{j}\land
    S^{i}_{m-1}(u_{1},\dots,\widehat{u_{j}},
    \dots,u_{m})
  \end{multline}
 and
  \begin{multline}
    \label{eq:4}
    \bar \partial
    S^{i}_{m}(u_{1},\dots,u_m)=(-2)^{m}(i-1)!(m-i+1)!(u_{1},
    \dots,u_{m})^{(i-1)}\\
    -(i-1)\sum_{j=1}^{m}
    (-1)^{j}(-2\partial\bar \partial)u_{j}\land
    S^{i-1}_{m-1}(u_{1},\dots,\widehat{u_{j}},
    \dots,u_{m}).
  \end{multline}
  Then, using the definition of the differential in the Deligne complex,
  \begin{align*}
    d_{\mathcal{D}}T_{m}(u_{1},\dots,u_m)&=
    \frac{1}{2m!}\sum_{i=1}^{m-1}(-1)^{i+1}\partial
    S^{i}_{m}(u_{1},\dots,u_m) +
    \frac{1}{2m!}\sum_{i=2}^{m}(-1)^{i+1}\bar \partial
    S^{i}_{m}(u_{1},\dots,u_m) \\
    &=\frac{(-2)^{m}}{2m!}\sum_{i=1}^{m-1}(-1)^{i+1}i!(m-i)!
      (u_{1},\dots,u_{m})^{(i)}\\
      &\phantom{AA}+\frac{(-2)^{m}}{2m!}\sum_{i=2}^{m}(-1)^{i+1}(i-1)!(m-i+1)!
      (u_{1},\dots,u_{m})^{(i-1)}\\
      &\phantom{AA}+\frac{1}{2m!}\sum_{i=1}^{m-1}(-1)^{i+1}(m-i)
      \sum_{j=1}^{m}(-1)^{j}(-2\partial\bar \partial)u_{j}\land
      S^{i}_{m-1}(u_{1},\dots,\widehat{u_{j}}
    \dots,u_{m}) \\
     &\phantom{AA}-\frac{1}{2m!}\sum_{i=2}^{m}(-1)^{i+1}(i-1)
      \sum_{j=1}^{m}(-1)^{j}(-2\partial\bar \partial)u_{j}\land
      S^{i-1}_{m-1}(u_{1},\dots,\widehat{u_{j}}
    \dots,u_{m})
    \\
    &=\frac{1}{2m!}\sum_{j=1}^{m}(-1)^{j-1}d_{\mathcal{D}}u_{j}\land
    \left(\sum_{i=1}^{m-1}(-1)^{i}(m-i+i)S^{i}_{m-1}(u_{1},\dots,\widehat{u_{j}}
    \dots,u_{m})
    \right)\\
    &=\sum_{j=1}^{m}(-1)^{j-1}d_{\mathcal{D}}u_{j}\bullet T_{m-1}(u_{1},\dots,\widehat{u_{j}}
    \dots,u_{m}).
  \end{align*}
  In order to prove \ref{item:2}, write temporarily
  \begin{displaymath}
      C_{m}(u_{1},\dots,u_{m})=\frac{1}{m!}\sum_{\sigma \in
        \mathfrak S_m} (-1)^{|\sigma|}u_{\sigma (1)}\bullet (
      \dots \bullet (u_{\sigma (m-1)}\bullet u_{\sigma (m)})).
    \end{displaymath}
    Then, it is easy to show by induction that the forms $C_{m}$ are a
    linear combination of monomials of the form $u_{\sigma
      (1)}\partial u_{\sigma (2)}\land\dots\land \partial u_{\sigma
      (i)}\land \bar \partial u_{\sigma (i+1)}\land\dots\land \bar
    \partial u_{\sigma (m)}$, for suitable integers $i$ and
    permutations $\sigma $. Since they are invariant under the action of
    the symmetric group they are a linear combination of the forms
    $S^{i}_{m}$. Say
    \begin{displaymath}
      C_{m}(u_{1},\dots,u_{m})=\sum_{i=1}^{m}\alpha
      _{i,m}S^{i}_{m}(u_{1},\dots,u_{m}).
    \end{displaymath}
    By the Leibniz rule the forms $C_{m}$ satisfy the relation
    (\ref{eq:3}). In particular, $d_{\mathcal{D}}C_{m}$ does not contain
    any term of the form $(u_{1},\dots,u_{m})^{(i)}$. By (\ref{eq:2})
    and (\ref{eq:4}) this implies that $\alpha _{i,m}=-\alpha
    _{i-1,m}$. Thus, we have to show that $\alpha
    _{1,m}=-1/(2m!)$.
    Given a differential form $\omega $ of degree $n$, we will denote
    \begin{displaymath}
      \bar
    F^{p}\omega  =\sum_{q\ge p}\omega  ^{(n-q,q)}.
    \end{displaymath}
    Since $F^{m-1}C_{m}=\alpha _{1,m}S^{1}_{m}$, to determine $\alpha
    _{1,m}$ we can compare $\bar \partial S^{1}_{m}$ with
    $\bar \partial \bar F^{m-1}C_{m}$. On the one hand
    \begin{displaymath}
      \bar \partial  S^{1}_{m}=(-2)^{m}m!(u_{1},
    \dots,u_{m})^{(i-1)}.
    \end{displaymath}
    On the other hand, since for $a\in \mathcal{D}^{p}(A,p)$
    and $b\in \mathcal{D}^{q}(A,q)$, it holds
    \begin{displaymath}
      \bar \partial F^{p+q-1}(a\bullet b)=-2\bar \partial\bar
      F^{p-1}a\land\bar \partial\bar F^{q-1}b,
    \end{displaymath}
    we obtain
    \begin{displaymath}
      \bar \partial F^{m+1}C_{m}(u_{1},\dots,u_{m})=
      (-2)^{m-1}(u_{1},
    \dots,u_{m})^{(i-1)}.
    \end{displaymath}
    Therefore $\alpha _{1,m}=-1/(2m!)$, which concludes the proof of
    the proposition.
\end{proof}
Given the inclusion, for $m\ge 1$,  $\mathcal{D}^{m}(A,m)\subset
A^{m-1}_{\R}(m-1)$, we can view $T_{m}(u_{1},\dots,u_{m})$ as an
element of $A^{m-1}_{\R}(m-1)$.
By the same techniques as in the proof of the previous proposition one can prove
\begin{prop}\label{prop:1} For $m>1$,
 the following equation holds
  \begin{multline*}
    d T_{m}(u_{1},\dots,u_{m})= 2^{m-1}\big((u_{1},\dots,u_{m})^{(m)}+
    (-1)^{m-1}(u_{1},\dots,u_{m})^{(0)}\big)\\
    +2 \sum_{i=1}^{m}(-1)^{i-1}\partial\bar \partial u_{i}\land T(u_{1},\dots,
      \widehat {u_{i}},\dots ,u_{m}).
  \end{multline*}
  For $m=1$, the following equation holds
  \begin{equation*}
    d T_{1}(u_{1})=d u_{1}=\partial u_{1}+\bar \partial u_{1}
    =(u_{1})^{(1)}+(u_{1})^{(0)}.
  \end{equation*}
\ \hfill $\square$
\end{prop}

Let now $X$ be a proper complex algebraic manifold, $Y$ a closed
integral subvariety of $X$ of codimension $p$, $\iota: \widetilde
Y\longrightarrow X$ a resolution of singularities of $Y$ and
$\mathbb C(Y)=\mathbb C(\widetilde Y)$ the function field of $Y$.  For
$f\in \mathbb C(Y)^{\times }=\mathbb C(Y)-\{0\}$, we write
$$\mathfrak{g}(f)=\frac{-1}{2}\log f\bar f\in
\mathcal{D}^{1}_{\log}(\widetilde Y\setminus \DDV f,1).$$ This is a
Green form on $\widetilde Y$ for the cycle $\DDV f$. More precisely
\begin{equation}
  \label{eq:21}
  d_{\mathcal{D}}\mathfrak{g}(f)=
  d_{\mathcal{D}}\Big(\frac{-1}{2}\log f\bar f\Big)-\delta _{\DDV f}=
  -\delta _{\DDV f}.
\end{equation}
 Then, for
$f_{1},\dots,f_{m}\in \mathbb C(Y)^{\times }$ and $1\le i\le m$,
we denote
\begin{multline} \label{eq:11}
S_m^i(f_1, \ldots , f_m)= S_m^i(\mathfrak{g}(f_1), \ldots ,
\mathfrak{g}(f_m))\\
=\sum_{\sigma \in \mathfrak S_m}
(-1)^{|\sigma|}\log |f_{\sigma(1)}|^2
\frac{df_{\sigma(2)}}{f_{\sigma(2)}}\wedge \cdots \wedge
\frac{df_{\sigma(i)}}{f_{\sigma(i)}}\wedge
\frac{d\overline{f}_{\sigma(i+1)}}{\overline{f}_{\sigma(i+1)}}
\wedge \cdots \wedge
\frac{d\overline{f}_{\sigma(m)}}{\overline{f}_{\sigma(m)}},
\end{multline}
and
\begin{equation}
  \label{eq:10}
T_m(f_1, \ldots , f_m)=\frac{1}{2m!}\sum_{i=1}^m(-1)^i
S_m^i(f_1, \ldots , f_m).
\end{equation}
This is a differential form on $\widetilde Y$, and has logarithmic
singularities
along $\Div (f_1)\cup \cdots \cup \Div (f_m)$. It is always
locally integrable because, when
 $\Div (f_1), \cdots , \Div (f_m)$ have common components, the
 graded-commutativity of the product $\bullet$ assures us that the
 possible non locally integrable terms cancel each other.
Note that, although now the definitions of $S^{i}_{m}$ and $T_{m}$ are
overloaded, there is no possible confusion. In definitions
\eqref{eq:12} and \eqref{eq:13}  the arguments are elements of a
Deligne algebra. By contrast, in definitions \eqref{eq:11} and
\eqref{eq:10} the arguments are rational functions.

We will denote the current on $X$ associated to $T_{m}$ by
\begin{equation}
  \label{eq:1}
  [T_{m}](f_1, \ldots , f_m)=\iota_{\ast}[T_{m}(f_1, \ldots , f_m)].
\end{equation}
This current belongs to $\mathcal{D}_{D}^{2p+m}(X,p+m)$. Be aware of
the conventions of \S 2.2 concerning the current associated to a
differential form.

%For
%any $0\leq i \leq n$, consider the differential form on $(\P^1)^n$
%$$S_n^i:= \sum_{\sigma\in \mathfrak{S}_n}(-1)^{\sigma}\log |z_{\sigma(1)}|^2
%\frac{dz_{\sigma(2)}}{z_{\sigma(2)}}\wedge \cdots \wedge
%\frac{dz_{\sigma(i)}}{z_{\sigma(i)}}\wedge \frac{d\bar
%z_{\sigma(i+1)}}{\bar z_{\sigma(i+1)}}\wedge \cdots \wedge
%\frac{d\bar z_{\sigma(n)}}{\bar z_{\sigma(n)}}.$$ This is a
%differential form with logarithmic singularities along the
%hyperplanes $x_i=0$ and $y_i=0$. Therefore, $ S_n^i\in
%\mmD^{n}_{\log}((\C^*)^n,n).$

%Following \cite{Takeda}, after theorem 1.5, the forms $S_n^i$  can be described
%as
%$$\small S_n^i=\frac{(i-1)!}{(n-i)!}\sum_{\alpha=1}^n (-1)^{\alpha+1}\log|z_{\alpha}|^2( d\log|z_{1}|^2\wedge
%\cdots \wedge \widehat{d\log|z_{i}|^2}\wedge \cdots \wedge d\log|z_{n}|^2
%)^{(i-1,n-i)}
%$$
%where $(i-1,n-i)$ denotes the part of degree $(i-1,n-i)$.

\subsection{Goncharov's differential forms.}
For any $f_1, \ldots , f_m$ rational functions on $X$, Goncharov has
defined in \cite{Goncharov} differential forms $r_{m-1}(f_1, \ldots ,
f_m)$ as follows:
\begin{eqnarray*}
r_{m-1}(f_1, \ldots , f_m) & = & (-1)^{m}
\sum_{0\leq 2j+1\leq m}
c_{j,m}\ALT_m\left(\begin{array}{r}\log |f_1|d\log |f_2|\wedge \cdots \wedge
\log |f_{2j+1}|\wedge \\ \wedge di\arg f_{2j+2}\wedge
\cdots \wedge di\arg f_m \end{array}\right).
\end{eqnarray*}
Here  the symbol $\underset{0\leq 2j+1\leq m}{\sum }$ means
the sum over integers $j$ such that $0\leq 2j+1\leq m$, $c_{j,m}$ are
the rational numbers
$$
c_{j,m}=\frac{1}{(2j+1)!(m-2j-1)!},
$$
and $\ALT_m$ stands for the alternating sum over the symmetric
group $\mathfrak S_m$, i.e.,
$$
\ALT_m(F(f_1, \ldots , f_m))=\sum_{\sigma \in \mathfrak S_m}
(-1)^{|\sigma|}F(f_{\sigma (1)}, \ldots , f_{\sigma (m)}).$$

\begin{obs}
The sign $(-1)^m $ appears due to the difference in sign on the
differential of the Deligne complex as considered here and as
considered by Goncharov
in \cite{Goncharov}.
\end{obs}

\vskip 1pc
\begin{thm} Goncharov's form $r_{m-1}(f_1, \ldots , f_m)$ agrees
  with Wang's form $T_m(f_1, \ldots , f_m)$.
\end{thm}
\begin{proof}
Since
$d\log |f|=\frac{1}{2}(\frac{df}{f}+\frac{d\overline{f}}{\overline{f}})$
and
$di\arg f=\frac{1}{2}(\frac{df}{f}-\frac{d\overline{f}}{\overline{f}})$,
the $(i-1, m-i)$-part of the form
$$
\ALT_m\left(\log |f_1|d\log |f_2|\wedge \ldots \wedge
d\log |f_{2j+1}|\wedge di\arg f_{2j+2}\wedge
\ldots \wedge di\arg f_m\right)
$$
is equal to
$$
\frac{1}{2^m}\sum_{k=0}^{m-i}\binom{2j}{k}
\binom{m-2j-1}{m-i-k}(-1)^{m-i-k}S_m^i(f_1, \ldots , f_m).
$$
Hence
\begin{align*}
r_{m-1}&(f_1, \ldots , f_m)=\frac{1}{2^m}
\sum_{0\leq 2j+1\leq m}\sum_{i=1}^m\sum_{k=0}^{m-i}c_{j,m}
\binom{2j}{k}\binom{m-2j-1}{m-i-k}(-1)^{i+k}
S_m^i(f_1, \ldots , f_m)  \\
%=&\frac{1}{2^m}\sum_{i=1}^m\sum_{k=0}^{m-i}\sum_{k\leq 2j\leq k+i-1}
%\frac{(-1)^{i+k}}{(2j+1)k!(2j-k)!(m-i-k)!(i+k-2j-1)!}
%S_m^i(f_1, \ldots , f_m)  \\
=&\frac{1}{2^m}\sum_{i=1}^m\sum_{k=0}^{m-i}
\frac{(-1)^{i+k}}{k!(m-i-k)!}\sum_{k\leq 2j\leq k+i-1}
\frac{1}{(2j+1)(2j-k)!(i+k-2j-1)!}S_m^i(f_1, \ldots , f_m).
\end{align*}
By Lemma \ref{lemmafactorials} we have
\begin{align*}
r_{m-1}&(f_1, \ldots , f_m) =\frac{1}{2^m}\sum_{i=1}^m\sum_{k=0}^{m-i}\frac{(-1)^{i+k}}{k!(m-i-k)!}
\left(\sum_{l=0}^k\frac{(-1)^lk!2^{i-1+l}}{(k-l)!(i+l)!}\right)
S_m^i(f_1, \ldots , f_m)  \\
=&\frac{1}{2^m}\sum_{i=1}^m\sum_{l=0}^{m-i}
\frac{(-1)^l2^{i-1+l}}{(i+l)!}\left(\sum_{k=l}^{m-i}
\frac{(-1)^{i+k}}{(k-l)!(m-i-k)!}\right)S_m^i(f_1, \ldots , f_m)  \\
=&\frac{1}{2^m}\sum_{i=1}^m\sum_{l=0}^{m-i}
\frac{(-1)^i 2^{i-1+l}}{(i+l)!}\left(\sum_{k=0}^{m-i-l}
\frac{(-1)^{k}}{k!(m-i-k-l)!}\right)S_m^i(f_1, \ldots , f_m).
\end{align*}
Note that
$$n!\sum_{k=0}^n \frac{(-1)^k}{k! (n-k)!}=\sum_{k=0}^n (-1)^k \binom{n}{k} =\begin{cases}0, &\ n>0,  \\ 1, &\ n=0. \end{cases}$$ This follows from the equation
$(1-x)^{n}=\sum_{k=0}^{n}(-1)^k \binom{n}{k}  x^k $, taking $x=1$.
%
%The equation
%$$
%(1-X)^{m-i-l}=\sum_{k=0}^{m-i-l}
%\frac{(-1)^k(m-i-l)!}{k!(m-i-k-l)!}X^k
%$$
%implies that
%$$
%\sum_{k=0}^{m-i-l}\frac{(-1)^{m-i-k-l}}{k!(m-i-k-l)!}
%=\begin{cases}0, &\ m-i-l>0,  \\ 1, &\ m-i-l=0. \end{cases}
%$$
Therefore
$$
r_{m-1}(f_1, \ldots , f_m)=\frac{1}{2^m}\sum_{i=1}^m
\frac{(-1)^{i}2^{m-1}}{m!}S_m^i(f_1, \ldots , f_m)=
T_m(f_1, \ldots , f_m).
$$
\end{proof}

\begin{lemma}\label{lemmafactorials}
For every pair of integers $0\leq q\leq p$, we have
$$
\sum_{q\leq 2j\leq p}\frac{1}{(2j+1)(2j-q)!(p-2j)!}=
\sum_{l=0}^q\frac{(-1)^lq!\ 2^{p-q+l}}{(q-l)!(p-q+l+1)!}.
$$
\end{lemma}
\begin{proof}
We denote the left hand side of the equation above by $A(q, p)$. The statement will be shown by induction on $q$.
When $q=0$,
$$
A(0, p)=\sum_{0\leq 2j\leq p}\frac{1}{(2j+1)!(p-2j)!}=
\frac{1}{(p+1)!}\sum_{0\leq 2j\leq p}\binom{p+1}{2j+1}.
$$
Since
\begin{equation}\label{poli}\frac{1}{2}\big((1+x)^{p+1}-(1-x)^{p+1}\big)=\sum_{0\leq 2j\leq p}
\binom{p+1}{2j+1}x^{2j+1},
\end{equation}
we obtain that (taking $x=1$)
$$A(0, p)=\frac{2^p}{(p+1)!}$$ as desired. Let us assume that the statement is true for $q-1$ and
let us now differentiate $q$ times  equation \eqref{poli}.
$$
\frac{1}{2}\big((1+x)^{p+1}-(1-x)^{p+1}\big)^{(q)} = \frac{(p+1)p\cdots (p-q+2)}{2}\big((1+x)^{p-q+1}-
(-1)^q(1-x)^{p-q+1}\big).
$$
On the other side, writing $B=\Big(\sum_{0\leq 2j\leq p}\binom{p+1}{2j+1}x^{2j+1}\Big)^{(q)}$, we have
\begin{eqnarray*}
B & = & \sum_{q-1\leq 2j\leq p}(2j+1)2j\cdots (2j-q+2)\binom{p+1}{2j+1}x^{2j-q+1} \\
&=& q\sum_{q-1\leq 2j\leq p}2j\cdots (2j-q+2)\binom{p+1}{2j+1}x^{2j-q+1} \\ && \hskip 4pc +
\sum_{q\leq 2j\leq p} 2j\cdots (2j-q+1)\binom{p+1}{2j+1}x^{2j-q+1}
\\
&=& \sum_{q-1\leq 2j\leq p} \frac{q(p+1)!}{(2j+1)(2j-q+1)!(p-2j)!}x^{2j-q+1}\\ &&  \hskip 4pc + \sum_{q\leq 2j\leq p} \frac{(p+1)!}{(2j+1)(2j-q)!(p-2j)!}x^{2j-q+1}.
\end{eqnarray*}
where we decomposed the sums using $2j+1=q+2j-q+1$.
%Using the decomposition $\sum_{0\leq 2j\leq p}\binom{p+1}{2j+1}x^{2j+1} = x\sum_{0\leq 2j\leq p}
%\binom{p+1}{2j+1}x^{2j},$ we obtain
%\begin{multline*}
%B=\big(\sum_{0\leq 2j\leq p}\binom{p+1}{2j+1}x^{2j}\big)^{(q-1)} + x
%\big(x\sum_{1\leq 2j\leq p}2j\binom{p+1}{2j+1}x^{2j-1}\big)^{(q-1)}
%\\ = \sum_{q-1\leq 2j\leq p} \frac{q(p+1)!}{(2j+1)(2j-q+1)!(p-2j)!}X^{2j-q+1}  \\
%%&\hskip 4pc+X\sum_{q\leq 2j\leq p}
%%\frac{(p+1)!}{(2j+1)(2j-q)!(p-2j)!}X^{2j-q}.
%
%\end{multline*}
%
%\begin{multline*}
%\big(x\sum_{0\leq 2j\leq p}\binom{p+1}{2j+1}x^{2j}\big)^{(q)} =
%\big(\sum_{0\leq 2j\leq p}\binom{p+1}{2j+1}x^{2j}\big)^{(q-1)} + x
%\big(x\sum_{1\leq 2j\leq p}2j\binom{p+1}{2j+1}x^{2j-1}\big)^{(q-1)}
%
% q\sum_{q-1\leq 2j\leq p}(2j)(2j-1)\cdots (2j-q+2)
%\binom{p+1}{2j+1}X^{2j-q+1}  \\
%&\hskip 4pc+X\sum_{q\leq 2j\leq p}(2j)(2j-1)\cdots
%(2j-q+1)\binom{p+1}{2j+1}X^{2j-q}  \\
%=&\ q\sum_{q-1\leq 2j\leq p}
%\frac{(p+1)!}{(2j+1)(2j-q+1)!(p-2j)!}X^{2j-q+1}  \\
%&\hskip 4pc+X\sum_{q\leq 2j\leq p}
%\frac{(p+1)!}{(2j+1)(2j-q)!(p-2j)!}X^{2j-q}.  \\
%\end{multline*}
Putting $x=1$ and joining the two sides of equation \eqref{poli}, we obtain
$$
(p+1)p\cdots (p-q+2)2^{p-q}=q(p+1)!A(q-1, p)+(p+1)!A(q, p).
$$
Hence, using the induction hypothesis, we have
\begin{eqnarray*}
A(q, p)&=&\frac{2^{p-q}}{(p-q+1)!}-qA(q-1, p) =\frac{2^{p-q}}{(p-q+1)!}-\sum_{l=0}^{q-1}\frac{q(-1)^l(q-1)!\ 2^{p-q+1+l}}{(q-1-l)!(p-q+l+2)!} \\
&=& \frac{2^{p-q}}{(p-q+1)!}+\sum_{l=1}^{q-1}\frac{(-1)^l q!\ 2^{p-q+l}}{(q-l)!(p-q+l+1)!}
=\sum_{l=0}^q\frac{(-1)^lq!2^{p-q+l}}{(q-l)!(p-q+l+1)!}.
\end{eqnarray*}
\end{proof}

\subsection{Relation of currents}
The assignment defined by $T_{m}$ can be seen as a morphism of abelian
groups
$$
T_m:\wedge^m\mathbb C(Y)^{\times }\longrightarrow
\mathcal{D}^{m}(\widetilde Y\setminus \mathcal{Z}^{1},m),
$$
whereas the assignment defined by $[T_{m}]$ gives a morphism of
abelian groups
$$
[T_m]:\wedge^m\mathbb C(Y)^{\times }\longrightarrow
\mathcal{D}_{D}^{2p+m}(X,p+m),
$$

We are now interested in differentiating the current $[T_m]$.
Assume $Y$ is normal and $Z\subset Y$ is a closed integral
subscheme of codimension one.
Then we can define the residue map
$$
\RES_Z:\wedge^m\mathbb C(Y)^{\times }\longrightarrow
\wedge^{m-1}\mathbb C(Z)^{\times }
$$
by means of the valuation of $\mathbb C(Y)$ with respect to $Z$.
For an arbitrary closed integral subscheme $Y$ of $X$, we define
the current $([T_{m-1}]\circ \RES )(f_1, \ldots , f_m)$ on $X$ by
$$
([T_{m-1}]\circ \RES )(f_1, \ldots , f_m)=
\sum_{Z\in \overline {Y}^{(1)}}(\iota_{Z})_{\ast}[
T_{m-1}\left(\RES_Z(\pi^*f_1, \ldots , \pi^*f_m)\right)],
$$
where $\overline{Y}$ is the normalization of $Y$, $\overline{Y}^{(1)}$
is the set of irreducible closed
subvarieties of codimension one on $\overline Y$ and
$\iota_{Z}:\widetilde Z\longrightarrow X$ is the composition of a
resolution of singularities of $Z$ with the natural map to $X$.

Then the differential of $T_m(f_1, \ldots , f_m)$ is described
as follows (\cite{Goncharov}, Proposition 2.8):

\begin{prop}\label{prop:dif}
Let $f_{1},\dots,f_{m}\in \mathbb{C}(Y)$. Then,
as differential forms on $\widetilde Y$ with logarithmic singularities,
we have
\begin{equation}
  \label{eq:5}
  dT_{m}(f_{1},\dots,f_{m})=\frac{(-1)^{m}}{2}\left(
    \frac{df_{1}}{f_{1}}\land\dots\land \frac{df_{m}}{f_{m}}+
    (-1)^{m-1}\frac{d\bar f_{1}}{\bar f_{1}}\land\dots\land
    \frac{d\bar f_{m}}{\bar f_{m}}
  \right),
\end{equation}
and
\begin{equation}
  \label{eq:6}
  d_{\mathcal{D}}T_{m}(f_{1},\dots,f_{m})=0.
\end{equation}
As currents on $X$ we have, for $m>1$,
\begin{multline}\label{eq:7}
  d[T_{m}](f_{1},\dots,f_{m})=\frac{(-1)^{m}}{2}\iota_{\ast}\left[
    \frac{df_{1}}{f_{1}}\land\dots\land \frac{df_{m}}{f_{m}}+
    (-1)^{m-1}\frac{d\bar f_{1}}{\bar f_{1}}\land\dots\land
    \frac{d\bar f_{m}}{\bar f_{m}}
  \right]\\
  + ([T_{m-1}]\circ \RES)(f_{1},\dots,f_{m}),
\end{multline}
for $m=1$,
\begin{equation}
  \label{eq:15}
   d[T_{1}](f_{1})=\frac{-1}{2}\iota_{\ast}\left[
    \frac{df_{1}}{f_{1}}+
    \frac{d\bar f_{1}}{\bar f_{1}}  \right],
\end{equation}
and
\begin{equation}\label{eq:8}
  d_{\mathcal{D}}[T_{m}](f_{1},\dots,f_{m})=
  -([T_{m-1}]\circ \RES)(f_{1},\dots,f_{m}).\end{equation}
\end{prop}
\begin{proof}
  This is proved in \cite{Goncharov} \S 2. Equations \eqref{eq:5}
  and \eqref{eq:6} follow directly from Propositions \ref{prop:pd} and \ref{prop:1}, using that $\partial\bar
  \partial \log f\bar f=0$ for any holomorphic function $f$. When
  $\DDV f_{1}\cup
  \dots \cup \DDV f_{m}$ is a
  normal crossing divisor, and these divisors do not have any common
  components, then equations \eqref{eq:7} and \eqref{eq:8} follow
  from the same propositions  using (\ref{eq:21}).
  The general case can be reduced to this one by
  using resolution of singularities.
\end{proof}

There are two main examples of the construction of this section that we
are interest in. The first one is the original definition of Wang's
forms, that are tailored to the cubical setting.

Let $\C^*=\P^1\setminus \{0,\infty\}$. If $(x_i:y_i)$ are
projective coordinates on the $i$-th projective line in
$(\P^1)^m$, let $f_i=y_i/x_i$ be the rational function on $(\P^1)^m$.
Then, Wang's forms defined in \cite{Wang} (see also
\cite{BurgosWang}) are given by
\begin{equation}\label{burgoswang}
W_m=T_m(y_1/x_1,\dots,y_m/x_m)\in \mmD^{m}_{\log}((\C^*)^m,m).
\end{equation}
In particular $W_{0}=1\in \mmD^{0}_{\log}(\Spec(\C),0)$.
We denote by $[W_{n}]\in \mmD^{m}_{D}((\P^1)^m,m)$ the associated
current, which we call Wang's current. In this case Proposition \ref{prop:dif} leads
\begin{thm}\label{thm:wdif} For every $m\ge 1$, Wang's currents
  satisfy the relation
  \begin{displaymath}
    d_{\mathcal{D}}[W_{m}]=\sum
    _{i=1}^{m}\sum_{j=0,1}(-1)^{i+j}(\delta ^{i}_{j})_{\ast} [W_{m-1}],
  \end{displaymath}
  where the maps $\delta ^{i}_{j}$ are the structural maps of the
  cubical structure of $(\P^1)^m$.
\end{thm}
\begin{proof}
  For $i=1,\dots,m$ and $j=0,1$, let $D^{i}_{j}$ be the divisor image
  of the structural map $\delta ^{i}_{j}$. Then the only non zero
  residues of  $\frac{y_{1}}{x_{1}}\land\dots\land
  \frac{y_{m}}{x_{m}}$ are
  \begin{displaymath}
    \RES_{D^{i}_{j}}\Big(\frac{y_{1}}{x_{1}}\land\dots\land
    \frac{y_{m}}{x_{m}}\Big)=
    (-1)^{i+j+1}\Big(\frac{y_{1}}{x_{1}}\land\dots\land
    \widehat{\frac{y_{i}}{x_{i}}}\land\dots\land
    \frac{y_{m}}{x_{m}}\Big)\Big|_{D^{i}_{j}}.
  \end{displaymath}
  Therefore, the result follows directly from Proposition \ref{prop:dif}.
\end{proof}

Wang's forms have another property that will be useful when
establishing the convergence of certain integrals.
\begin{prop}\label{prop:vanish} Let $D^{m}$ be the divisor
  $D^{m}=(\P^{1})^{m}\setminus \square ^{m}$ introduced in \S
  \ref{BFregulator}. Let $f:X\longrightarrow
  (\C^*)^{m}$ be a holomorphic map that factors through $D^{m}\cap
  (\C^*)^{m}$. Then
  \begin{displaymath}
    f^{\ast}W_{m}=0.
  \end{displaymath}
\end{prop}
\begin{proof}
  This follows from the definition of $W_{m}$, because
$ D^{m}=\bigcup_{i=1}^{m}\{x_{i}=y_{i}\} $   and $\log(1)=0$.
\end{proof}

The second example is the original definition of Goncharov, that is
tailored to the simplicial setting. Recall that we have fixed
homogeneous coordinates $z_0,\ldots,z_n$ of $\mathbb P^n$ and we
denoted
$\Delta^n=\mathbb P^n\setminus H_n,$ where $H_n$ is the hyperplane
defined by
$z_0+\cdots +z_n=0$.

We denote
\begin{equation}
  \label{eq:16}
  G_{m}=T_{m}\Big(\frac{z_{1}}{z_{0}},\dots,\frac{z_{m}}{z_{0}}\Big)
\end{equation}
and let $[G_{m}]$ be the associated current (called Goncharov's current).

In this case Proposition \ref{prop:dif} leads
\begin{thm}\label{thm:gdif} Goncharov's currents satisfy the relation
  \begin{displaymath}
    d_{\mathcal{D}}[G_{m}]=\sum
    _{i=0}^{m}(-1)^{i}(\partial ^{i})_{\ast} [G_{m-1}],
  \end{displaymath}
  where the maps $\partial ^{i}$ are the structural maps of the
  semi-simplicial structure of $\P^m$.
\end{thm}
\begin{proof}
  Let $D_{i}\subset \P^{m}$ denote the divisor of equation
  $z_{i}=0$. The result follows directly from the relations
  \begin{displaymath}
    \RES_{D_{0}}\left(\frac{z_{1}}{z_{0}}
    \land\dots\land\frac{z_{m}}{z_{0}}\right)
    =\RES_{D_{0}}\left(\frac{z_{1}}{z_{0}}
    \land\frac{z_{2}}{z_{1}}\land\dots\land\frac{z_{m}}{z_{1}}\right)
    =-\left(\frac{z_{2}}{z_{1}}\land\dots\land\frac{z_{m}}
      {z_{1}}\right)\Big|_{D_{0}}
  \end{displaymath}
  and
  \begin{displaymath}
    \RES_{D_{i}}\left(\frac{z_{1}}{z_{0}}
    \land\dots\land\frac{z_{m}}{z_{0}}\right)
    =(-1)^{i-1}\left(\frac{z_{1}}{z_{0}}
    \land\dots\land\widehat{\frac{z_{i}}{z_{0}}}\land\dots\land
    \frac{z_{m}}{z_{0}}\right)\Big|_{D_{i}}.
  \end{displaymath}
\end{proof}

%\vskip 1pc
%{\it Remark}:
%In \cite{Goncharov} Goncharov defines $T_m(f_1, \ldots , f_m)$
%in a different manner, and uses
%the other notation $r_{m-1}(f_1, \ldots , f_m)$.
%But since it can be proved that these two currents coincide
%with each other, we adopt the former notation,
%which comes from \cite{BurgosWang}.

\section{Algebraic cycles and the Beilinson regulator}\label{regulator2}
In this section we compare the regulator defined by Goncharov to its
cubical version, and show that the cubical
version agrees with the Beilinson regulator, by comparing it to the
construction given by Burgos and Feliu in \cite{BFChow}.

Throughout this section, $X$ will be an equidimensional compact complex
algebraic manifold.

\medskip
\subsection{Goncharov's conjecture} Let $\mathcal D_{D}^{*}(X,p)$ be
the Deligne complex of currents of \S \ref{sec:currents}. We denote by
 $\tau \mathcal D_{D}^{*}(X,p):=\tau_{\le 2p} \mathcal D_{D}^{*}(X,p)$
 the truncated complex.

For each integer $m\ge 0$, let $\pi_X:X\times \Delta^m \rightarrow X$
be the projection onto the first factor. To simplify the notation we
will use the same symbol for these morphisms. In each case it will be
clear which one is used.
Let $z_0,\dots,z_m$ be projective coordinates of
$\Delta^m=\P^m\setminus H_{m}$. For a closed integral subscheme
$Z\subset X\times \Delta^m$ of codimension $p$ which intersects
properly with each face of $ X\times \Delta^m$, let $\overline Z$ be
the Zariski closure of $Z$ on $X\times \P^{m}$ and let
$\iota:\widetilde Z\longrightarrow X\times \P^{m}$ be the composition
of a resolution of singularities of $\overline Z$ with the inclusion
of $\overline Z$ in $ X\times \P^{m}$. We define the current
$\mathcal P_{s}^p(m)(Z)\in
\tau \mathcal D_{D}^{2p-m}(X,p)$ by
\begin{eqnarray*}
\mathcal P_{s}^p(m)(Z) & = & (\pi_X)_* \iota_{\ast}
\left[T_m\left(\left.\frac{z_1}{z_0}\right|_{\widetilde Z}, \ldots ,
\left.\frac{z_m}{z_0}\right|_{\widetilde Z}\right)\right] \\
& = &(\pi_X)_*(\delta_Z\wedge G_m),
\end{eqnarray*}
where $\delta_Z$ is the current integration along $Z$.
Let $$
\mathcal P_{s}:Z_s^{p}(X, m)\longrightarrow
\tau \mathcal D_{D}^{2p-m}(X,p)
$$
be defined by $\mmP_s(Z)=\mathcal P_{s}^p(m)(Z)$ if $Z$ is as above,
and extended to cycles
$z\in Z_s^{p}(X, *)$ by linearity. Observe that if $m=0$,
$\mathcal{P}_{s}(z)=\delta _{z}$ is a closed current and therefore it
belongs to the truncated complex. Remember that
we are including the twist in the definition of the current associated
to a differential form and the definition of the current associated to
an algebraic cycle.

Theorem \ref{thm:gdif} implies that
$d_{\mathcal D}\left(\mathcal P^p(m)(z)\right)
=\mathcal P^p(m-1)(\partial z)$ (see also
\cite{Goncharov}, Theorem 2.12).
Hence, we have the following result.
\begin{lemma}
  The morphism $\mmP_s$ is a chain morphism.
\end{lemma}

Goncharov has presented the following conjecture:

\vskip 1pc
\begin{conj} \label{goncharovconj} Let $X$ be an  equidimensional
  compact complex algebraic manifold.
The composition
$$ K_n(X)_{\Q} \xrightarrow{\cong} \bigoplus_{p\geq 0}
CH^{p}_{s}(X,n)_{\Q} \xrightarrow{\mathcal P_s} H_{\mathcal D}^{2p-n}(X,p)
$$
agrees with Beilinson's regulator.
\end{conj}

\subsection{Cubical construction}
We introduce here the cubical version of Goncharov's regulator using
Wang's forms.

Let $\pi_X:X\times \square^m \rightarrow X$ be here the projection
onto the first factor. Let $(x_i:y_i)$ be homogeneous coordinates on
the $i$-th factor of $(\P^{1})^{m}$.
 For a closed integral subscheme
$Z\subset X\times \square^m$ of codimension $p$ which intersects
properly with each face of $ X\times \square^m$, let $\overline Z$ be
the Zariski closure of $Z$ on $X\times (\P^{1})^{m}$ and let
$\iota:\widetilde Z\longrightarrow X\times (\P^{1})^{m}$ be the composition
of a resolution of singularities of $\overline Z$ with the inclusion
of $\overline Z$ in $ X\times (\P^{1})^{m}$. We define the current
$\mathcal W^{p}(m)(Z)\in
\mathcal \tau D_{D}^{2p-m}(X,p)$ by
\begin{eqnarray*}
\mathcal W^p(m)(Z) & = & (\pi_X)_* \iota_{\ast}
\left[T_m\left(\left.\frac{y_1}{x_1}\right|_{\widetilde Z}, \ldots ,
\left.\frac{y_m}{x_m}\right|_{\widetilde Z}\right)\right] \\
& = &(\pi_X)_*(\delta_Z\wedge W_{m}).
\end{eqnarray*}
This gives a map
\begin{equation}
Z^p_c(X,m)_0  \xrightarrow{\mmP_c}  \tau \mmD_{D}^{2p-m}(X,p)
\label{pc}
\end{equation}
defined by $\mmP_c(Z)=\mathcal{W}^{p}(m)(Z)$,  if $Z$ is as above,
and extended to cycles
$z\in Z_c^{p}(X, *)_0$ by linearity.

%if $Z$ is an irreducible codimension $p$ subvariety of $X\times \square^n$ intersecting properly the faces of
%$\square^n$.
%Observe that we have \textcolor{blue}{ (s\'{\i}???)}
%$$\pi_{X*}(\delta_Z \wedge T_n) = \frac{1}{2\pi i}\pi_{X*}((T_n)_{\mid Z}) $$

\begin{lemma}\label{pcmodcom}
  The morphism $\mmP_c$ is a chain morphism.
\end{lemma}
\begin{proof}
This follows from Theorem \ref{thm:wdif}.
\end{proof}
We denote also by
$$\mmP_{c}: CH^p_c(X,n) \longrightarrow H_{\mmD}^{2p-n}(X,\R(p))$$
the induced morphism.

\subsection{Comparison of $\mathcal{P}_{c}$ and $\rho $} We will now
compare
the map
$\mmP_{c}$ with Beilinson's regulator.
\begin{thm}\label{qinverse}
  \begin{enumerate}
  \item \label{item:4} Given any differential form $\alpha \in
    \mathcal{D}^{r,-m}_{\A}(X,p)$, the form $\alpha\bullet W_{m}$ is
    locally integrable as a singular form on $X \times
    (\P^{1})^{m}$. Hence it defines a current $[\alpha\bullet
    W_{m}]\in \mathcal{D}_{D}^{r+m}(X\times
    (\P^{1})^{m},p+m)$. Moreover
    \begin{displaymath}
      d_{\mathcal{D}}[\alpha \bullet W_{m}]=
      [d_{\mathcal{D}}\alpha\bullet W_{m}]
      +(-1)^{r}\sum_{j=0,1}\sum_{i=1}^{m}(-1)^{i+j}(\delta
      ^{i}_{j})_{\ast}[(\delta
      ^{i}_{j})^{\ast}\alpha \bullet W_{m-1}]
    \end{displaymath}
  \item \label{item:5} The assignment $\alpha \longmapsto
    (\pi_{X})_{\ast}[\alpha\bullet
    W_{m}]$ defines a morphism of complexes
    \begin{displaymath}
      \mmD_{\A}^*(X,p)\xrightarrow{\varphi} \tau \mmD_{D}^{*}(X,p).
    \end{displaymath}
    Hence, by composition, a morphism of complexes
    \begin{displaymath}
       \mmD_{\A}^*(X,p)_{0}\xrightarrow{\varphi} \tau \mmD_{D}^{*}(X,p).
    \end{displaymath}
  \item \label{item:6} If we identify $\tau \mmD^{*}(X,p)$ with a
    subcomplex of $\tau \mmD_{D}^{*}(X,p)$ via the morphism (\ref{eq:14}),
    then the image of
    $\varphi $ is contained in $\tau \mmD^{*}(X,p)$. By abuse of notation
    we will also denote by $\varphi$ the induced morphism
    \begin{displaymath}
      \mmD_{\A}^*(X,p)_0 \xrightarrow{\varphi} \tau \mmD^{*}(X,p).
    \end{displaymath}
  \item \label{item:7} The morphism $\varphi$ is a quasi-inverse of
    the quasi-isomorphism given in Proposition
    \ref{difaffine},~\ref{item:8}.
  \end{enumerate}
\end{thm}
\begin{proof}
 Statements \ref{item:4} and \ref{item:6} follow easily from Proposition
  \ref{prop:vanish} and the definition of $ \mmD_{\A}^*(X,p)_0 $, by
  using the techniques of \cite{Burgos:Gftp} \S 3.
  We next prove that $\varphi$ is a morphism of complexes. Let $\alpha \in
  \mmD_{\A}^{r,-m}(X,p)$. Then
  \begin{align*}
    \varphi(d \alpha )&=\varphi(d_{\mathcal{D}}\alpha)+
    (-1)^{r}\sum_{j=0,1}\sum_{i=1}^{m}(-1)^{i+j}\varphi((\delta
    ^{i}_{j})^{\ast}\alpha )\\
    &=(\pi _{X})_{\ast}[d_{\mathcal{D}}\alpha\bullet W_{m}]
    +(-1)^{r}\sum_{j=0,1}\sum_{i=1}^{m}(-1)^{i+j}(\pi _{X})_{\ast}([(\delta
    ^{i}_{j})^{\ast}\alpha \bullet W_{m-1}])\\
    &=(\pi _{X})_{\ast}(d_{\mathcal{D}}([\alpha\bullet W_{m}])
    )\\
    &=d_{\mathcal{D}}\varphi(\alpha ).
  \end{align*}
  For statement \ref{item:7}, if we denote by $\psi $ the
  quasi-isomorphism of  Proposition
    \ref{difaffine},~\ref{item:8}, then, by definition, the
    composition $\varphi\circ
    \psi $ is the identity.
\end{proof}

Let $(\omega ,g)\in
\mathcal{D}^{2p,-m}_{\A,\mathcal{Z}^{p}}(X,p)_{0}$.
Since we have defined
$\mathcal{D}^{\ast}_{\A,\mathcal{Z}^{p}}(X,p)_{0}$ using truncated
complexes, the pair $(\omega ,g)$ is closed. Moreover, by the purity
property of Deligne cohomology, there exists a cycle $z\in
Z^{p}(X,m)_{0}\otimes \R$ such that $[(\omega ,g)]=\cl(z)$. Let
$\delta _{z}$ be the associated current. Since the set where $g$ is
singular
has codimension $p$, \cite{Burgos:Gftp} Corollary 3.8
implies that $g$ is locally integrable on $X\times
\square^{m}$. Then, $\omega$ and $g$ determine currents on
the Deligne complex $$\mathcal{D}^{\ast}(\mathscr{D}_{X\times
  (\P^{1})^{m}/D^{m}}(X\times  (\P^{1})^{m}),p),$$
where $\mathscr{D}_{X\times
  (\P^{1})^{m}/D^{m}}$ is the sheaf of currents defined, for instance
in \cite{BurgosKuhnKramer} after Definition 5.43. Moreover, by
adapting the proof of \cite{Burgos:Gftp} Theorem 4.4, to the above
complex, one can prove that they satisfy the equation of currents
\begin{equation}
  \label{eq:17}
  d_{\mathcal D}[g]+\delta _{z}=[\omega].
\end{equation}

Using again the techniques of the proof of \cite{Burgos:Gftp} Theorem
4.4 one obtains
\begin{prop}\label{qinvsing1}
Let $(\omega ,g)\in
\mathcal{D}^{2p,-m}_{\A,\mathcal{Z}^{p}}(X,p)_{0}$.
Then, the differential form $g\bullet
  W_{m}$ is locally integrable  as a form on $X\times
  (\P^{1})^{m}$. Moreover,
  \begin{displaymath}
    d_{\mathcal{D}}[g\bullet W_{m}]=
    [\omega\bullet W_{m}]-\delta _{z}\bullet W_{m}
    -[\delta g \bullet W_{m-1}].
  \end{displaymath}
\ \hfill $\square$
\end{prop}

Let now $(\omega ,g)\in
\mathcal{D}^{r,-m}_{\A,\mathcal{Z}^{p}}(X,p)_{0}$, with $r<2p$.
Again, since the set where $g$ is
singular
has codimension $p$, \cite{Burgos:Gftp} Corollary 3.8
implies that $g$ is locally integrable on $X\times
\square^{m}$. Then, $\omega$ and $g$ determine currents on
the Deligne complex $$\mathcal{D}^{\ast}(\mathscr{D}_{X\times
  (\P^{1})^{m}/D^{m}}(X\times  (\P^{1})^{m}),p).$$
Moreover they satisfy the equations of currents
\begin{equation}
  d_{\mathcal D}[g]=[d_{\mathcal{D}}g],\quad d_{\mathcal{D}}[\omega ]=
  [d_{\mathcal{D}}\omega ].
\end{equation}
and we have

\begin{prop}\label{qinvsing2}
Let $(\omega ,g)\in
\mathcal{D}^{r,-m}_{\A,\mathcal{Z}^{p}}(X,p)_{0}$, with $r<2p$.
    Then, the differential form $g\bullet
  W_{m}$ is locally integrable  as a form on $X\times
  (\P^{1})^{m}$. Moreover,
  \begin{displaymath}
    d_{\mathcal{D}}[g\bullet W_{m}]=
    [d_{\mathcal{D}}g\bullet W_{m}]+(-1)^{r-1}
    [\delta g \bullet W_{m-1}].
  \end{displaymath}
\ \hfill $\square$
\end{prop}

Let $\mathcal{D}_{\A,\mathcal{H}}^{2p-\ast}(X,p)_{0}$ be the complex
\eqref{eq:18}. Then the central result of this subsection is the following.
\begin{thm}\label{thm:morcomp} The map
  \begin{displaymath}
    \psi :\mathcal{D}_{\A,\mathcal{H}}^{2p-\ast}(X,p)_{0}
    \longrightarrow \tau \mathcal{D}_{D}^{2p-\ast}(X,p)
  \end{displaymath}
  given by
  \begin{displaymath}
    \psi (z,(\omega ,g),\alpha )=\mathcal{P}_{c}(z)-
    (\pi _{X})_{\ast}[g\bullet W_{m}]+\varphi(\alpha ),
  \end{displaymath}
  when $(\omega ,g)\in
  \mathcal{D}^{r,-m}_{\A,\mathcal{Z}^{p}}(X,p)_{0}$, is a morphism of
  complexes. Moreover, there is a commutative diagram
  \begin{displaymath}
    \xymatrix{
      Z^{p}_{c}(X,\ast)_{0}\ar[r]^{\mathcal{P}_{c}}\ar[d]^{\rho }
      & \tau \mathcal{D}_{D}^{2p-\ast}(X,p)\ar@{=}[d]\\
      \mathcal{D}^{2p-\ast}_{\A,\mathcal{H}}(X,p)_{0}\ar[r]^{\psi}_{\sim}
      & \tau \mathcal{D}_{D}^{2p-\ast}(X,p)\\
      \mathcal{D}^{2p-n}_{\A}(X,p)_{0}\ar[u]_{\beta
      }^{\sim}\ar[r]_{\sim}^{\varphi}
      & \tau \mathcal{D}^{2p-\ast}(X,p)\ar[u]^{\sim}
    }
  \end{displaymath}
\end{thm}
\begin{proof}
  If $z\in \mathcal{H}^{p}(X,\ast)_{0}$, then Lemma \ref{pcmodcom}
  implies that $d\psi ((z,0,0))=\psi (d(z,0,0))$. If $\alpha \in
  \mathcal{D}_{\A}^{2p-\ast}(X,p)_{0}$, then Theorem \ref{qinverse}
  implies that $d\psi((0,0,\alpha ))=\psi (d(0,0,\alpha ))$.

  Let $(\omega ,g)\in
  \mathcal{D}^{2p,-m}_{\A,\mathcal{Z}^{p}}(X,p)_{0}$. Let  $z\in
Z^{p}(X,m)_{0}\otimes \R$ such that $[(\omega ,g)]=\cl(z)$, and let
$\delta_{z}$ be the associated current. Then, using Proposition
\ref{qinvsing1}, we have
\begin{align*}
  d\psi ((0,(\omega ,g),0)&=-d_{\mathcal{D}}(\pi _{X})_{\ast}[g\bullet
  W_{m}]\\
  &=(\pi _{X})_{\ast}[\delta _{z}\bullet W_{m}]+(\pi _{X})_{\ast}[\delta
  g\bullet W_{m-1}]-(\pi _{X})_{\ast}[\omega \bullet W_{m}],
\end{align*}
and
\begin{align*}
  \psi (d(0,(\omega ,g),0))&=\psi ([(\omega ,g)],-(\delta \omega
  ,\delta g),-\omega )\\
  &=(\pi _{X})_{\ast}[\delta _{z}\bullet W_{m}]+(\pi _{X})_{\ast}[\delta
  g\bullet W_{m-1}]-(\pi _{X})_{\ast}[\omega \bullet W_{m}].
\end{align*}
If $(\omega ,g)\in
  \mathcal{D}^{r,-m}_{\A,\mathcal{Z}^{p}}(X,p)_{0}$, with $r<2p$,
  using Proposition \ref{qinvsing2}, we have
  \begin{align*}
    d\psi ((0,(\omega ,g),0)&=-d_{\mathcal{D}}(\pi _{X})_{\ast}[g\bullet
    W_{m}]\\
    &=-(\pi _{X})_{\ast}[d_{\mathcal{D}}g\bullet W_{m}]
    +(-1)^{r}(\pi _{X})_{\ast}[\delta g\bullet
    W_{m-1}],
  \end{align*}
and
\begin{align*}
    \psi (d(0,(\omega ,g),0))&=\psi ((0,(-d_{\mathcal{D}}\omega
    ,-\omega +d_{\mathcal{D}}g)-(-1)^{r}(\delta \omega
  ,\delta g),-\omega )\\
  &=-(\pi _{X})_{\ast}[d_{\mathcal{D}}g\bullet W_{m}]
    +(-1)^{r}(\pi _{X})_{\ast}[\delta g\bullet
    W_{m-1}].
\end{align*}
The fact that the diagram is commutative follows directly from the
definition of the maps involved.
\end{proof}

As an immediate consequence of Theorem \ref{thm:morcomp} we have

\begin{thm}\label{rhop} For all $n,p\geq 0$ the morphisms
  $$\mmP_{c},\rho: CH^p_c(X,n) \longrightarrow H_{\mmD}^{2p-n}(X,\R(p))$$
  agree.
\end{thm}
\vist

\subsection{Proof of the conjecture} At this point, we have seen that
the cubical version of Goncharov's construction agrees with the
regulator defined by Burgos-Feliu. In this section we prove the
Goncharov's conjecture \ref{goncharovconj}, by showing that the
cubical and the simplicial constructions agree. This will be done
through an intermediate complex consisting of both simplicial and
cubical affine schemes.

Let $X$ be an equidimensional quasi-projective algebraic scheme of
dimension $d$ over the field $k$. By a face of $X\times
\square^{n}\times \Delta ^{m}$ we understand any subset of the form
$X\times F\times G$, where $F$ is a face of $\square^{n}$ and $G$ is a
face of $\Delta ^{m}$.

Let $Z^{p}_{cs}(X,n,m)$ be the free abelian
group generated by the codimension $p$ closed irreducible subvarieties
of $X\times \square^{n}\times \Delta^m$, which intersect properly all
the faces of $X\times \square^n\times \Delta^m$. Then
$Z^{p}_{cs}(X,n,m)$  has a simplicial structure with faces
$\partial_{i}$ and a cubical structure with faces $\delta_i^j$.
 Let $Z_{cs}^p(X,n,m)_0$ be
the subgroup of $Z^{p}_{cs}(X,n,m)$ consisting of those elements that
lie in the kernel of $\delta_i^1$ for all $i=1,\dots,n$.

Consider the 2-iterated chain complex $Z_{cs}^p(X,*,*)_0$ whose piece
of degree $(n,m)$ is $Z_{cs}^p(X,n,m)_0$ and whose differentials are
$(\delta,\partial)$. Let $Z_{cs}^p(X,*)_0$ denote the simple complex
associated to $Z_{cs}^p(X,*,*)_0$.

\begin{prop}[\cite{Levine1}, Theorem 4.7]\label{isic}
  The natural morphisms
  $$    Z_{s}^p(X,*) \xrightarrow{i_s} Z_{cs}^p(X,*)_0, \qquad
  Z^p_c(X,*)_0  \xrightarrow{i_c} Z_{cs}^p(X,*)_0,
  $$
  are both quasi-isomorphisms.\hfill$\square$
\end{prop}
This result is the key to show that the higher Chow groups defined
using the cubical or the simplicial version agree. Moreover,
it follows that the higher algebraic Chow groups can also be computed
in terms of the complex $Z_{cs}^p(X,*)_0$,
that is $CH^p(X,n) \cong H_n(Z_{cs}^p(X,*)_0).$
As usual, we denote
$$CH^p_{cs}(X,n) = H_n(Z_{cs}^p(X,*)_0).$$

Assume that $X$ is an equidimensional projective complex algebraic manifold.
The strategy pursued to prove the conjecture is the following. We will construct
a regulator map
$$CH^p_{cs}(X,n) \xrightarrow{\mmP_{cs}} H^{2p-n}_{\mmD}(X,\R(p))$$
and show that there is a big commutative square
\begin{equation}\label{strategy}
\xymatrix{
 CH_{s}^p(X,n)  \ar[dd]^{\cong}_{i_c^{-1}i_s}  \ar[dr]^{\cong}_{i_s}
 \ar[drrr]^{\mmP_s} \\
 & CH^p_{cs}(X,n) \ar[rr]^{\mmP_{cs}} && H^{2p-n}_{\mmD}(X,\R(p)) \\
 CH_{c}^p(X,n) \ar[ur]_{\cong}^{i_c}\ar[urrr]_{\mmP_c}
}
\end{equation}
Then, since  $\mmP_c$ is Beilinson's regulator, it will
follow that so are $\mmP_s$ and $\mmP_{cs}$.

\medskip
\textbf{The morphism $\mmP_{cs}$.}
Recall again the projective coordinates $(x_{i}:y_{i})$ on the
$i$-th projective line in $\square^n\subset (\P^1)^n$ and
the homogeneous coordinates $(z_0:\dots:z_m)$ on $\Delta^m\subset \P^m$.
We denote
\begin{equation}
  M_{n,m}=T_{n+m}\Big(\frac{y_{1}}{x_{1}},\dots,\frac{y_{n}}{x_{n}},
  \frac{z_{1}}{z_{0}},\dots,\frac{z_{m}}{z_{0}}\Big),
\end{equation}
which is a differential form on $\square^n\times \Delta^m$ with
logarithmic singularities along
$\partial \square^n\times \Delta^m\cup \square^n\times \partial \Delta^m$.
In particular, $M_{n,0}=W_n$ and $M_{0,m}=G_m$.
Consider the group morphism
\begin{eqnarray}
Z^p_{cs}(X, n, m)_0 &  \xrightarrow{\mmP_{cs}}  &  \tau\mmD_{D}^{2p-m}(X,p)
\label{pcs} \\
Z &\mapsto & (\pi_{X})_{*}(\delta_Z \wedge M_{n,m}) \nonumber
\end{eqnarray}
if $Z$ is an irreducible codimension $p$ subvariety of $X\times
\square^n \times \Delta^m$ intersecting properly the faces of   $X\times
\square^n \times \Delta^m$, analogous to the definition of $\mmP_c,\mmP_s$.

%\textcolor{blue}{Any signs to change?}

\begin{lemma}
  The morphism
  $$  Z_{cs}^p(X,*)_0 \xrightarrow{\mmP_{cs}}  \tau\mmD_{D}^{2p-*}(X,p)$$
  is a chain morphism.
\end{lemma}
\begin{proof}
Denote by $[M_{n,m}]$ the current on $(\P^1)^n\times \P^m$ associated
to $M_{n,m}$.
By Proposition \ref{prop:dif}, one can prove the equation of currents
\begin{displaymath}
    d_{\mathcal{D}}[M_{n,m}]=\sum_{i=1}^{n}\sum_{j=0,1}(-1)^{i+j}
    (\delta ^{i}_{j}\times 1)_{\ast} [M_{n-1,m}]+(-1)^n
    \sum_{i=0}^{m}(-1)^{i}(1\times \partial ^{i})_{\ast} [M_{n,m-1}].
  \end{displaymath}
The result follows easily from this relation.
\end{proof}

The proof of next lemma is straightforward.

\begin{lemma}\label{commutativediagrams}
  The diagrams
$$  \begin{array}{ccc}
\xymatrix@R=5pt{
     Z_{s}^p(X,*) \ar[dd]^{\sim}_{i_s} \ar[drr]^{\mmP_s} \\ & &
     \tau\mmD^{2p-*}_{D}(X,p) \\ Z_{cs}^p(X,*)_0 \ar[urr]_{\mmP_{cs}}} &
   \qquad &
\xymatrix@R=5pt{
     Z_{c}^p(X,*)_0 \ar[dd]^{\sim}_{i_c} \ar[drr]^{\mmP_c} \\ & &
    \tau \mmD^{2p-*}_{D}(X,p) \\ Z_{cs}^p(X,*)_0 \ar[urr]_{\mmP_{cs}}}
\end{array} $$
are commutative.
  \end{lemma}
  \vist

\begin{thm}
  Let $X$ be an  equidimensional projective complex algebraic
  manifold. Let $\mmP'_s$ be the composition of $\mmP_s$ with the
  isomorphism
given by the Chern character of \cite{Bloch1}
$$\mmP_s': K_n(X)_{\Q} \xrightarrow{\cong} \bigoplus_{p\geq 0} CH^{p}(X,n)_{\Q} \xrightarrow{\mmP_s}
 \bigoplus_{p\geq 0}H^{2p-n}_{\mmD}(X,\R(p)).$$
Then, the morphism $\mmP_s'$ agrees with  Beilinson's regulator.
\end{thm}
\begin{proof}
  From  Lemma \ref{commutativediagrams} we see
  that there is a commutative diagram \eqref{strategy}. Then, the
  statement
  follows from Theorem \ref{beichow} together with Theorem \ref{rhop}
\end{proof}

\section{Higher arithmetic Chow groups}
In this last section, we use the comparison of regulators performed in
the previous sections to show that the higher arithmetic Chow groups
given by Goncharov agree with the ones given by Burgos and Feliu, for
all proper arithmetic varieties.

Following \cite{GilletSouleIHES}, by an \emph{arithmetic variety} $X$
over a ring $A$ we mean a
regular scheme $X$ which is flat and quasi-projective over an
arithmetic ring $A$.

%\subsection{Goncharov higher arithmetic Chow groups}
Assume that $X$ is a smooth proper variety defined over an
arithmetic field $F$.
Then we obtain the associated real variety
$X_{\R}=(X_{\mathbb C}, F_{\infty })$, that is, a projective complex
algebraic manifold $X_{\mathbb C}$ equipped with
an anti-holomorphic involution $F_{\infty }$.
We write
\begin{equation}
  \label{eq:20}
   \mmD^{\ast}(X,p)=
 \mmD_{D}^{\ast}(X_{\mathbb{C}},p)^{\overline{F}_{\infty }=\IId},\qquad
 \mmD_{D}^{\ast}(X,p)=
 \mmD_{D}^{\ast}(X_{\mathbb{C}},p)^{\overline{F}_{\infty }=\IId }.
\end{equation}

Then we have the regulator map for $X$:
$$
\mathcal P_s: Z_s^p(X,n)\longrightarrow
 \tau\mmD_{D}^{2p-m}(X,p).
$$
Let $Z\mathcal{D}^{2p}(X,p)$ be the space of cycles of degree $2p$ in the
Deligne complex of smooth differential forms on $X$, considered as a
chain complex concentrated in degree zero.

\begin{df}[Goncharov]
Let
 {\small $$
\widehat{Z}_G^p(X, *)=
s\left(\begin{array}{c}\xymatrix@C=0.5pt{
& \tau\mathcal{D}^{2p-\ast}_{D}(X,p) &\\
Z^{p}_{s}(X,\ast)\ar[ur]^{\mathcal{P}_{s}}&&
Z\mathcal{D}^{2p}(X,p)\ar[ul]_{[\cdot]}
} \end{array}\right)
$$ }
be the simple of the diagram.
The \emph{higher arithmetic Chow groups} of $X$ are defined to be
the homology groups of this complex:
$$
\widehat{CH}_{G}^p(X,n)=H_n(\widehat{Z}_G^p(X, *)).
$$
\end{df}

%\subsection{Burgos-Feliu higher arithmetic Chow groups}
We define
$$\mmD^{\ast}_{\A,\mathcal{H}}(X,p)_{0}:=
\mmD^{\ast}_{\A,\mathcal{H}}(X,p)_{0}
^{\overline{F}_{\infty}^*=id}.$$

We denote by $\beta :Z\mathcal{D}^{2p}(X,p)\longrightarrow \mathcal{D}^{2p-\ast}_{\A,\mathcal{H}}(X,p)_{0}$ the map
given by $\beta (\alpha )=(0,0,\alpha )$.

\begin{df}[Burgos-Feliu]
Let
 {\small $$ \widehat{Z}_{BF}^p(X, *)=
s\left(\begin{array}{c}\xymatrix@C=0.5pt{
& \mathcal{D}^{2p-\ast}_{\A,\mathcal{H}}(X,p)_{0} &\\
Z^{p}_{c}(X,\ast)_{0}\ar[ur]^{\rho }&&
Z\mathcal{D}^{2p}(X,p)\ar[ul]_{\beta } } \end{array}\right) $$
}
be the simple of the diagram.
The \emph{higher arithmetic Chow groups} of $X$ as defined in \cite{BFChow} are
the homology groups of this complex:
$$
\widehat{CH}_{BF}^p(X,n)=H_n(\widehat{Z}_{BF}^p(X, *)).
$$
\end{df}

\begin{thm}\label{chowagree} For all $n,p\ge 0$ there are
  natural isomorphisms
  $$\widehat{CH}_{G}^p(X,n)\longrightarrow \widehat{CH}_{BF}^p(X,n).$$
\end{thm}
\begin{proof}
  Let $\widehat{CH}_{c}^p(X,n)$ and $\widehat{CH}_{cs}^p(X,n)$ be the
  analogues of $\widehat{CH}_{G}^p(X,n)$ defined using the cubical and
  the cubical-simplicial setting given by the morphisms $\mmP_c$ and $\mmP_{cs}$. It follows from Lemma \ref{commutativediagrams} and Proposition \ref{isic} that there are natural isomorphisms
  \begin{displaymath}
    \widehat{CH}_{G}^p(X,n)\xrightarrow{\cong}\widehat{CH}_{cs}^p(X,n)
    \xleftarrow{\cong}\widehat{CH}_{c}^p(X,n).
  \end{displaymath}
By the commutative diagram of Theorem \ref{thm:morcomp} there are natural isomorphisms
  \begin{displaymath}
    \widehat{CH}_{BF}^p(X,n)\xrightarrow{\cong}\widehat{CH}_{c}^p(X,n),
  \end{displaymath}
  and the theorem is proved.
  \end{proof}

As a consequence, we can transfer properties from one definition of
higher arithmetic Chow
groups to the other. In particular we obtain the following result.

\begin{cor} Let $X$ be a projective arithmetic variety over an
  arithmetic field and let $\widehat{CH}_G^p(X,n)$ denote the higher
  arithmetic Chow
groups defined by Goncharov.
\begin{itemize}
\item \emph{(Pull-back)}: Let $f:X\rightarrow Y$ be a morphism between
  two projective arithmetic
varieties over a field. Then, there is a pull-back morphism
$$\widehat{CH}^p_G(Y,n) \xrightarrow{f^*} \widehat{CH}^p_G(X,n), $$ for every
$p$ and $n$, compatible with the pull-back maps on the groups
$CH^p(X,n)$ and $H_{\mmD}^{2p-n}(X,\R(p))$.
\item \emph{(Product)}: There exists a product on
$$\widehat{CH}^*_G(X,*):=\bigoplus_{p\geq 0,n\geq 0}
\widehat{CH}^p_G(X,n), $$ which is associative, graded commutative
with
respect to the degree $n$.
\end{itemize}
  \end{cor}
\begin{proof}
It follows from Theorem \ref{chowagree} together with the results in
\cite{BFChow}, where these properties are shown for
$\widehat{CH}_{BF}^p(X,n)$ .
\end{proof}

%\bibliographystyle{alpha}
%\bibliography{bib2}

\end{document}